\documentclass[a4paper,10pt]{article}
\usepackage{amsmath,amssymb,latexsym,amsthm}
\usepackage{array,graphicx,epsfig,xcolor}
\usepackage{hyperref}
\usepackage{stmaryrd}

\newcommand{\ds}{\displaystyle}

\newcommand{\Z}{\mathbb{Z}}
\newcommand{\Q}{\mathbb{Q}}
\newcommand{\R}{\mathbb{R}}
 
\newcommand{\N}{\mathbb{N}} 
\newcommand{\T}{\mathbb{T}}
\newcommand{\I}{\mathcal{I}}

\newcommand{\norm}[1]{\left\Vert#1\right\Vert}
\newcommand{\A}{{\mathcal{A}}}
\newcommand{\B}{{\mathcal{B}}}
\newcommand{\Bcal}{{\mathcal{B}}}
\newcommand{\X}{{\mathcal{X}}}
\newcommand{\U}{{\mathcal{U}}}
\newcommand{\F}{{\mathcal{F}}}

\newtheorem{theorem}{Theorem}[section]
\newtheorem{definition}[theorem]{Definition}

\newtheorem{proposition}[theorem]{Proposition}

\newtheorem{remark}[theorem]{Remark}

\numberwithin{equation}{section}

\title{Transmutation techniques and observability for time-discrete approximation schemes of conservative systems\thanks{The first author is partially supported by the Agence Nationale de la Recherche (ANR, France), Project CISIFS number NT09-437023, and grant MTM2011-29306 of the MICINN, Spain. Part of this work has been done while he was visiting the BCAM -- Basque Center for Applied Mathematics   as a Visiting Fellow.  The second author is  supported by  the ERC Advanced Grant FP7-246775 NUMERIWAVES,  the Grant PI2010-04 of the Basque Government, the ESF Research Networking Program OPTPDE and Grant MTM2011-29306 of the MICINN, Spain.}}
\author{Sylvain Ervedoza$^{1,2,}$\footnote{e-mail: {\tt ervedoza@math.univ-toulouse.fr}}\\
{\it \footnotesize $^{1}$ CNRS ; Institut de Math\'ematiques de Toulouse UMR 5219 ;}\\
{\it \footnotesize	 F-31062 Toulouse, France,}\\
{\it \footnotesize $^{2}$ Universit\'e de Toulouse ; UPS, INSA, INP, ISAE, UT1, UTM ; IMT ;}\\
{\it \footnotesize	 F-31062 Toulouse, France,}\\
and  Enrique Zuazua$^{3,4,}$
\footnote{e-mail: {\tt zuazua@bcamath.org}}\\
{\it  \footnotesize $^{3}$ BCAM - Basque Center for Applied Mathematics, Alameda de Mazarredo, 14}\\
{\it \footnotesize  E-48009 Bilbao, Basque Country, Spain.}\\
{\it  \footnotesize $^{4}$ Ikerbasque Research Professor, Ikerbasque - Basque Foundation for Science,}\\
{\it \footnotesize E-48011 Bilbao, Basque Country, Spain.}
}

\date\today
\begin{document}
\maketitle
\begin{abstract}
	In this article, we consider abstract linear conservative systems and their time-discrete counterparts. Our main result is a representation formula expressing solutions of the continuous model through the solution of the corresponding time-discrete one. As an application, we show how observability properties for the time continuous model yield uniform (with respect to the time-step) observability results for its time-discrete approximation counterparts, provided the initial data are suitably filtered. The main output of this approach is the estimate on the time under which we can guarantee uniform observability for the time-discrete models. Besides, using a reverse representation formula, we also prove that this estimate on the time of uniform observability for the time-discrete models is sharp. We then conclude with some general comments and open problems.
\end{abstract}
%
%
%

\section{Introduction}

\subsection{Setting} Assume that $\A$ is a skew-adjoint unbounded operator defined in a Hilbert space $\X$ with dense domain $\mathcal{D}(\A) \subset \X$ and compact resolvent. We consider the equation
\begin{equation}
		\label{FD-Abstract}
		y' = \A y, \quad t \in \R, \qquad y(0) = y^0 \in \X.
\end{equation}

This equation is well-posed for $t \in \R$ and, $\A$ being skew-symmetric, solutions $y(t)$ of \eqref{FD-Abstract} have constant norms $\norm{y(t)}_\X$. Often  $\norm{y(t)}_\X^2/2$ is referred to 	as the energy of the system. 

 Several classical systems fit this abstract setting, as for instance the wave equation, Schr\"odinger's and Maxwell's equations, among many others.

Our primary goal is to establish a link between the solutions $y$ of \eqref{FD-Abstract} and the solutions of time-discrete versions of \eqref{FD-Abstract}. Since the solutions $y$ of \eqref{FD-Abstract} have constant energy, we will reduce our analysis to time-discretization schemes which preserve the $\X$-norm $\norm{\cdot}_\X$ of the solutions.
\medskip

To state our results precisely, we need to introduce filtered spaces of solutions. For, we write down the spectral decomposition of the operator $\A$, which is given by a sequence of purely imaginary eigenvalues $(i \mu_{j})$, corresponding to an orthonormal basis of $\X$ constituted by the eigenvectors $\Phi_j$. For $\delta >0$, we define the filtered space
\begin{equation}
		\label{FiteredSpaceTime}
		\mathfrak{C}(\delta) = \hbox{span} \{\Phi_j  \hbox{ corresponding to eigenvalues } |\mu_j|\leq \delta \}.
\end{equation}
One easily checks that this space is left invariant by the equation \eqref{FD-Abstract}, and thus we will identify the space of trajectories $y$ solutions of \eqref{FD-Abstract} lying in $\mathfrak{C} (\delta)$ with the space of initial data $y^0$ lying in $\mathfrak{C}(\delta)$.
\medskip

Let us now describe the time-discretization schemes under consideration. 

We assume that the time-discretization of system \eqref{FD-Abstract} discretized with a time-step $\tau >0$ takes the form
\begin{equation}
		\label{AbstractTimeDisc}
		y^{k+1}_\tau = \T_\tau y^k_\tau, \quad k \in \Z, \qquad y^0_\tau = y^0.
\end{equation}
Here, $y^k_\tau$ denotes an approximation of the solution $y$ of \eqref{FD-Abstract} at time $k \tau$, and $\T_\tau$ is assumed to be an approximation of $\exp(\tau \A)$.

To be more precise, we assume that $\T_\tau$ is a linear operator which has the same eigenvectors as the operator $\A$, and such that, for some function $f$,
\begin{equation}
	\label{TonEig}
	\T_\tau \Phi_j = \exp (i \mu_{j, \tau} \tau) \Phi_j, \quad \hbox{where } \mu_{j, \tau} = \frac{1}{\tau} f (\mu_j \tau).
\end{equation}
Actually, this operator may also be written as
$$
	\T_\tau = \exp\left( i f(- i \A \tau)\right).
$$
Let us now make precise the assumptions we impose on $f$.

First, we assume that $f$ is $C^\infty$ and satisfies 
\begin{equation}
		\label{HypG-1}
	f(0) = 0, \quad f'(0) = 1.
\end{equation}
This assumption on $f$ is satisfied for most of the time-discretization approximation schemes: Roughly speaking, this is equivalent to the consistency of the time-discretization. 

To simplify notations and avoid technical developments, we furthermore assume that $f$ is odd. This is not necessary in our arguments, but in practice, most of the time-discretization schemes modeled by \eqref{AbstractTimeDisc} fit in this class. 

We also assume that $f$ is real valued and 
\begin{equation}
	\label{HypG-2}
	f : (-R, R) \to (-\pi, \pi), 
\end{equation} 
where $R \in \R_+^* \cup \{\infty\}$ may be infinite. The fact that $f$ is real-valued is equivalent to the fact that the norms of solutions of \eqref{AbstractTimeDisc} are constant in the discrete time $k \in \Z$ and ensures the stability of the numerical scheme. The range of $f$ is limited to $(-\pi, \pi)$ to avoid aliasing since one cannot measure oscillations at frequencies higher than $\pi/\tau$ on a discrete mesh of mesh-size $\tau$.

In particular, for \eqref{AbstractTimeDisc} to be well-defined, we always consider solutions of \eqref{AbstractTimeDisc} with initial data lying in some  class of filtered data at the scale $\tau$, namely $\mathfrak{C}(\delta/\tau)$ for some $\delta \in (0,R)$.

We also assume that for all $\delta < R$, 
\begin{equation}
	\label{HypG-3}
	 \inf\{f'(\alpha); \ |\alpha| \leq \delta \} >0
\end{equation}
to ensure the invertibility of $f$.

We finally define $g : (-f(R), f(R)) \to \R$ as the inverse function of $f$ on $(-R, R)$. This can be done due to assumption \eqref{HypG-3}.

Before going further, let us point out that several classical time-discretization schemes fit the abstract setting provided by Assumptions \eqref{HypG-1}--\eqref{HypG-2}--\eqref{HypG-3}, see Section \ref{ExamplesTimeDisc} for several examples.
\medskip

\subsection{A representation formula} We are now in position to state the following result, whose proof can be found in Section \ref{Sec-Rho-tau}:
\begin{theorem}\label{Thm-DiscTransmutation}
	Let $f$ be a smooth function describing the time discrete operator $\T_\tau$ as in \eqref{TonEig}, assume \eqref{HypG-1}--\eqref{HypG-2}--\eqref{HypG-3}, and fix $\delta \in (0,R)$.

	Let $\chi$ be a $C^\infty$ function compactly supported in $(-f(R), f(R))$ and equal to $1$ in $(-f(\delta), f(\delta))$. For $\tau >0$, define then $\rho_\tau(t,s)$ for $t \in \R,\ s \in \R$ as 
	\begin{equation}
		\label{KernelDiscTransmutation}
		\rho_\tau(t, s) = \frac{1}{2\pi} \int_{- \pi/\tau}^{\pi/\tau} \exp\left(-\frac{i g(\mu_\tau \tau) t}{\tau} \right) \chi\left(\mu_\tau\tau\right) e^{i \mu_\tau s}\, d\mu_\tau.
	\end{equation}
	Then, if $y^0 \in \mathfrak{C}(\delta/\tau)$ and $y^k_\tau$ is the corresponding solution of \eqref{AbstractTimeDisc},  the function $y(t)$ defined by 
	\begin{equation}
		\label{Transmutation}
		y(t) = \tau \sum_{k \in \Z} \rho_\tau(t, k \tau) y_\tau^k
	\end{equation}
	is the solution of \eqref{FD-Abstract} with initial data $y^0$.
\end{theorem}

\begin{remark}
Note that when $f$ is bijective from $(-R,R)$ to $(-\pi, \pi)$, the above representation formula does not require the use of the cut-off function $\chi$ and allows describing all solutions $y$ of \eqref{FD-Abstract} for all initial data in $y^0 \in \mathfrak{C}(R/\tau)$, see Remark \ref{Remark-a} for further details.
\end{remark}

We now give some informal arguments motivating the representation formula in Theorem \ref{Thm-DiscTransmutation}.

 First, we remark that under the assumptions on the operator $\A$, solutions $y$ of \eqref{FD-Abstract} admit the following spectral decomposition:
\begin{equation}
	\label{Y-expanded}
	y(t) = \sum_j a_j e^{i \mu_j t} \Phi_j, 
\end{equation}
where $(a_j)$ are the coefficients of the initial datum on the basis $\Phi_j$:
\begin{equation}
	\label{Init-Data}
	y^0 = \sum_j a_j \Phi_j.
\end{equation}
Similarly, due to \eqref{TonEig}, the solutions $y_\tau$ of the time-discrete system \eqref{AbstractTimeDisc} with initial datum $y^0$ as in \eqref{Init-Data} can be written as
\begin{equation}
	\label{Y-tau-Expanded}
	y^k_\tau = \sum_j a_j e^{i \mu_{j,\tau} k \tau} \Phi_j,
\end{equation}
where the $\mu_{j,\tau}$ are defined by \eqref{TonEig}. Obviously, we can then extrapolate a time-continuous version $z_\tau$ of $y_\tau$:
\begin{equation}
	\label{Z-cont}
	z_\tau(s) = \sum_j a_j e^{i \mu_{j,\tau} s} \Phi_j.
\end{equation}
Here, we use the notation $s$ for the continuous time-variable corresponding to the \emph{a priori} time-discrete dynamics, and $z_\tau$ for the new state variable to avoid confusion with time $t$ and state $y$ of the continuous dynamics. 

Transmutation or subordination refers to the possibility of expressing one semi-group as a function of another one. In our context, this simply consists in writing the solutions $y$ of \eqref{FD-Abstract} as a function of $z_\tau(s)$ by means of a kernel $\tilde \rho = \tilde \rho(t,s)$ under the form
\begin{equation}
	\label{GeneralTransForm}
	y(t) = \int_\R \tilde \rho_{\tau} (t,s) z_\tau(s) \, ds, \quad t \in \R.
\end{equation}
For this to be done, using the explicit expressions \eqref{Y-expanded} and \eqref{Z-cont}, the kernel $\tilde \rho(t,s)$ has to be built so that 
\begin{equation}
	\label{Formal-tilde-rho-req}
	e^{i \mu_j t} = \int_\R \tilde \rho_\tau(t,s) e^{i \mu_{j,\tau} s} \, ds, \quad t \in \R, \, \mu \in \R.
\end{equation}
Interpreting the right-hand side of \eqref{Formal-tilde-rho-req} as a Fourier transform in $s$ and taking into account that $\mu_{j,\tau} = f (\mu_j \tau)/\tau$ with $f$ invertible, naturally leads  to
\begin{equation}
	\label{Formal-tilde-rho}
	\tilde \rho_\tau(t,s) = \frac{1}{2 \pi} \int_\R \exp\left(-\frac{i g(\mu \tau) t}{\tau} \right) e^{i \mu s}\, d\mu.
\end{equation}
This formula is a simplified version of the one of $\rho_\tau$ in \eqref{KernelDiscTransmutation}, in which the above time-continuous function $z_\tau$ in \eqref{Z-cont} has to be replaced by the time-discrete function $y_\tau$ in \eqref{Y-tau-Expanded} and the filtering operator $\chi$ has been introduced.
\medskip

\subsection{Application to observability} As an application of Theorem \ref{Thm-DiscTransmutation}, we consider an observation problem for the equation \eqref{FD-Abstract} and its time-discrete counterparts for \eqref{AbstractTimeDisc}.  This problem, inspired in control theoretical issues, concerns the possibility of recovering  the full energy of solutions out of partial measurements. The question is relevant both in the continuous and in the time-discrete frame, and in the later, a natural question arising in numerical analysis is to know whether the observability property is uniform with respect to the time-step. Indeed, this problem is the dual version of the classical controllability one and the uniformity (with respect to time-step) of the observability inequality is equivalent to the convergence of numerical controls towards the continuous ones as the time-step tends to zero (see \cite{Zua05Survey,ErvZuaCime}).

We thus consider an observation operator $\Bcal$ taking value in some Hilbert space $\U$ and assumed to belong to $ \mathfrak{L}(\mathcal{D}(\A^p), \U)$ for some $p\in \N$. To be more precise, we assume that there exists $C_{p}>0$ such that 
\begin{equation}
	\label{C-p-boundedness}
		\norm{\Bcal y}_{\U}^2 \leq C_{p}^2 \left(\norm{\A^p y}_{\X}^2  + \norm{y}_{\X}^2 \right), \quad \forall y \in \mathcal{D}(\A^p).
\end{equation}
We furthermore assume that equation \eqref{FD-Abstract} is observable through $\Bcal$ at time $T_0>0$, meaning that there exists a constant $C_0$ such that, for all $y^0 \in \mathcal{D}(\A^p)$, the solution $y(t)$ of \eqref{FD-Abstract} with initial data $y^0$ satisfies
\begin{equation}
	\label{FD-Obs}
		\norm{y^0}_\X^2 \leq C_0 \int_0^{T_0} \norm{\Bcal y(t)}_\U^2 \, dt.
\end{equation}
Estimate \eqref{FD-Obs} is the so-called observability estimate for \eqref{FD-Abstract}. As it has been established in the classical works \cite{DoleckiRussell,Lions}, this property is essentially equivalent to the controllability of the adjoint system:
\begin{equation}
	\label{ControlledSystem}
	y_c' = \A y_c + \Bcal^* v, \quad t \in (0,T), \qquad y_c(0) = y_c^0, 
\end{equation}
where $v$ is a control function in $L^2(0,T; \U)$. Actually the control for \eqref{ControlledSystem} can be built by minimizing a suitable quadratic functional over the class of solutions \eqref{FD-Abstract}. We do not give details on these links in this article and we refer the interested reader to \cite{Lions,TWbook}. Also note that the observability property is also closely linked to some inverse problems, see e.g. the work \cite{AlvesSilvestreTakahashiTucsnak} for precise statements in a setting similar to ours and the references therein.

We are thus interested in deriving discrete versions of \eqref{FD-Obs} for solutions $y_\tau$ of \eqref{AbstractTimeDisc}. Namely, we are asking if, given $\delta \in (0,R)$, there exist a time $T>0$ and  constants $C>0$ and $\tau_0>0$ independent of $\tau>0$ such that for all $\tau \in (0,\tau_0)$, solutions $y_\tau$ of \eqref{AbstractTimeDisc} with initial data $y^0 \in \mathfrak{C}(\delta/\tau)$ satisfy
\begin{equation}
	\label{Fully-Obs}
		\norm{y^0}_\X^2 \leq C \tau \sum_{ k \tau \in (0,T)} \norm{\Bcal y_\tau^k}_\U^2.
\end{equation}
This is a discrete observability estimate for the time-discrete equations \eqref{AbstractTimeDisc} which is uniform with respect to the time-discretization parameter, provided the initial state is suitably filtered. As it is well-known, see e.g. the survey articles \cite{Zua05Survey,ErvZuaCime}, this is needed to derive  algorithms to compute convergent sequences of discrete controls approximating the control of the continuous dynamics.

Our second main result, whose proof is given in Section \ref{Sec-Rho-tau}, is the following one:

\begin{theorem}\label{Thm-Main}
	Assume that $\Bcal \in \mathfrak{L}(\mathcal{D}(\A^p), \U)$ for some $p\in \N$ and $\U$ an Hilbert space, and that $\Bcal$ satisfies \eqref{C-p-boundedness} with constant $C_p$.
	
	Assume that equation \eqref{FD-Abstract} is observable through $\Bcal$ at time $T_0$ with constant $C_0$, i.e. for all $y^0 \in \mathcal{D}(\A^p)$, the solution $y(t)$ of \eqref{FD-Abstract} with initial data $y^0$ satisfies \eqref{FD-Obs}.

	Let $f$ be a smooth function describing the time discrete operator $\T_\tau$ as in \eqref{TonEig}, assume \eqref{HypG-1}--\eqref{HypG-2}--\eqref{HypG-3}, and fix $\delta \in (0,R)$.
	Then, for all 
	\begin{equation}
		\label{TimeConditionTau}
		T > \frac{T_0}{\underset{|\alpha| \leq \delta}{\inf} \{f'(\alpha) \}},
	\end{equation} 
	there exist positive constants $C$ and $\tau_0>0$ depending on $f, T, \delta, p, C_p, C_0,T_0$ such that for any $\tau \in (0,\tau_0)$, solutions $y_\tau$ of \eqref{AbstractTimeDisc} lying in $\mathfrak{C}(\delta/\tau)$ satisfy \eqref{FD-Obs}.
\end{theorem}

Theorem \ref{Thm-Main} was already proved in \cite{je3} under an additional admissibility property for the operator $\Bcal$ (see \cite{Lions} for a definition, or Section \ref{Sec-Admissibility} in our context), but with an estimate on the observability time  which is worse than the one we prove here in \eqref{TimeConditionTau}. Indeed, the technique in \cite{je3}, based on a resolvent characterization of observability due to \cite{BurqZworski,Mil05},  does not yield explicit estimates on the time of observability. A variant of this strategy was developed in \cite{je9} for time-discrete approximations of Schr\"odinger equations in a geometric setting in which the corresponding wave equation is observable, where it was proved that the time-discrete approximations of Schr\"odinger equations  are uniformly observable in any time $T>0$. In that case indeed, the resolvent estimates behave much better at high-frequencies, see e.g. \cite{TWbook}, and the low frequency components of the solutions can be handled following the arguments of \cite{HarauxTime}. 
These resolvent estimates were also used to derive observability estimates for space semi-discrete conservative systems in filtered classes of initial data, see \cite{je8,je7,Miller2012}, but still with non-explicit estimates on the time of observability.

The new ingredient introduced in this article that allows us to improve the results in \cite{je3} is the representation formula \eqref{Transmutation} and  careful estimates on the kernel $\rho_\tau$ in \eqref{KernelDiscTransmutation}. This will be done using classical techniques of harmonic analysis, and in particular the oscillatory phase lemma.  In particular, Proposition~\ref{RhoDescription} shows that, for all $\varepsilon>0$, $\rho_\tau$ is polynomially small with respect to $\tau>0$ to any arbitrary order  in the set
$$
	\Big\{(t,s) \in (0,T) \times \R, \hbox{ s.t. } t + \varepsilon < s \underset{|\alpha| \leq \delta + \varepsilon}\inf \{f'(\alpha)\}  \hbox{ or } 	 s \sup_{|\alpha| \leq \delta + \varepsilon} \{f'(\alpha)\} < t-\varepsilon \Big\}.
$$
To give some insights on this result, let us again consider the informal arguments given above and remark that $z_\tau(s)$ in \eqref{Z-cont} formally solves
$$
	\partial_s^{\tau} z_\tau = A z_\tau,
$$
where $\partial_s^{\tau}$ is the operator defined on the Fourier basis by
\begin{equation}
	\label{Def-ds}
	\forall \mu \in \R, \quad \partial_s^\tau \left(e^{i \mu s}\right) = i \frac{g(\mu \tau)}{\tau} e^{i \mu s}.
\end{equation}
Thus, the kernel $\tilde \rho_\tau$ in \eqref{Formal-tilde-rho} formally satisfies the transport-like equation
\begin{equation}
	\label{Transport-Pseudo}
	\partial_t \tilde \rho_\tau - \partial_s^\tau \tilde \rho_\tau = 0, \quad (t,s) \in \R \times \R, 
\end{equation}
with an initial data $\rho_\tau(t = 0, s) = \delta_{0}(s)$ where $\delta_{0}(s)$ is the Dirac function, since, according to \eqref{Formal-tilde-rho}, $\tilde \rho_\tau(t= 0, \cdot)$ is the Fourier transform of the function taking value $1$ identically.

Note that the solution of $(\partial_t - \partial_s) \rho_* = 0$ with initial data $\rho_*(t = 0) = \delta_0(s)$  is simply the Dirac delta transported at velocity one: $\rho_*(t,s) = \delta_{t} (s)$. Of course, for the kernel under consideration, $\partial_s^\tau$ is not a classical differential operator, but it is nevertheless very close to the classical derivation operator $\partial_s$ at low frequency since $g'(0) = 1$. We can thus interpret \eqref{Transport-Pseudo} as a transport equation with some added dispersion effects at frequencies of the order of $1/\tau$.

Let us also point out that the representation formula in \eqref{Transmutation} given through the kernel $\rho_\tau$ in \eqref{KernelDiscTransmutation} is not far from a Fourier integral operator - see e.g. \cite{Treves-Book-80} -  with the phase $\varphi(t,s, \mu_\tau) = \mu_\tau s- g(\mu_\tau) t$, and similar techniques can be employed. In particular the above localization result can also be seen as a counterpart of the fact that the kernel of a Fourier integral operator with phase $\varphi$ is regularizing outside of the set in which $\partial_{\mu_\tau} \varphi = 0$.

Next, in Theorem \ref{DiscTransmutationReverse}, we discuss a reverse representation formula giving solutions of the time-discrete equation \eqref{AbstractTimeDisc} in terms of solutions of the continuous equation \eqref{FD-Abstract}. This allows us to prove admissibility results for solutions of \eqref{AbstractTimeDisc} uniformly with respect to the time-discretization parameter $\tau$, see Theorem \ref{MainFully5}, and the optimality of the time estimates in \eqref{TimeConditionTau}, see Section \ref{SectionOptimality}.

Let us also point out that, similarly as in \cite{je3}, our approach can also be applied in the context of fully-discrete schemes: Indeed, to derive observability estimates for fully discrete approximations of \eqref{FD-Abstract} that are uniform in both time and space discretization parameters, our approach shows that it is sufficient to prove observability estimates for time-continuous and space semi-discrete approximation schemes that are uniform in the space discretization parameter, see Section \ref{Sec-Further-Space} for precise statements. Also note that our results can be used to recover discrete Ingham inequalities in a slightly different setting as the one in \cite{NegZua06}, see Section \ref{Sec-Further-Ingham}. We also explain how our strategy applies in the case of weak observability estimates in Section \ref{Sec-Further-Weak}. 
\medskip

\subsection{Related results}
Our approach is inspired by several previous works which establish representation formula for solutions of one equation through the solution of another equation. As we have already said, this technique is called subordination (in particular in the context of functional analysis, see e.g. \cite{Pruss93} and references therein) or transmutation. For instance, the work \cite{Kannai77}  provides estimates on the heat kernel thanks to an analysis of the corresponding wave equation and the so-called Kannai transform expressing solutions of the heat equation in terms of solutions of the wave equation, and the work \cite{Kannai92} proposes a study of singular problems based on a representation formula adding one dimension to the problem.

In the context of control theory in which we focus here, the so-called Fourier Bros Iagnoniltzer transform, linking solutions of the wave equation to a suitable elliptic operator,  provides an efficient tool to prove quantified unique continuation properties, see e.g. \cite{Robbiano91,Robbiano95,Phung09}.

Similarly, using a suitable transformation linking the wave and heat equations, Miller in \cite{Miller06a,Miller06b} derived estimates on the cost of controllability of the heat equations in small time, later improved in \cite{TucsnakTenenbaumTransAMS}. Similar estimates, related to the characterization of the reachable set for heat equations, can be found in \cite{ErvZuazuaARMA}. The main difference between the transforms in \cite{Miller06a,Miller06b} and \cite{ErvZuazuaARMA} is that, whereas the articles \cite{Miller06a,Miller06b} express solutions of the heat equations in terms of solutions of the wave equation, \cite{ErvZuazuaARMA} is based on a transform expressing solutions of the wave equation in terms of the solutions of the heat equation. Note that the robustness of the transformation of \cite{ErvZuazuaARMA} is illustrated by the fact that the weak observation properties derived in \cite{Phung09} for the waves in general geometric setting (not necessarily satisfying the geometric control conditions) can be used to recover control properties for the heat operator. Let us also point out that these transmutation techniques can also be used to derive numerical schemes to compute approximate controls for the heat equations \cite{MunchZuazua2010}.

We also emphasize that Theorem \ref{Thm-Main} states observability results for the time-discrete schemes \eqref{TonEig} observed through an observation operator $\B$ provided the continuous model \eqref{FD-Abstract} is observable through $\B$. Hence the first step is to check the observability property for the continuous model, and these observability properties have to be checked in each situation. When the continuous model \eqref{FD-Abstract} stands for a wave equation and $\B$ is a distributed observation operator or a boundary observation, the necessary and sufficient condition for observability is the so-called geometric control condition, see \cite{Bardos,BurqGerard}. For what concerns plate or Schr\"odinger's equations, the geometry still plays a role, but due to the infinite speed of propagation the situation is more intricate: for these models, observability holds when the geometric control condition is satisfied (see e.g. \cite{Miller2004} where this result is derived using transmutation techniques), but other less restrictive geometric settings still enjoy observability properties, see e.g. \cite{BurqZworski}.

\subsection{Outline} The article is organized as follows. In Section \ref{ExamplesTimeDisc}, we give several instances of classical time-discretization schemes that fit the abstract setting \eqref{TonEig} with a function $f$ satisfying Assumptions \eqref{HypG-1}--\eqref{HypG-2}--\eqref{HypG-3}. Section \ref{Sec-Rho-tau} is devoted to the proofs of Theorem \ref{Thm-DiscTransmutation} and Theorem \ref{Thm-Main}. In Section \ref{Sec-Reverse} we present a reverse representation formula and discuss its application to uniform admissibility results and use it to prove the optimality of the time estimate \eqref{TimeConditionTau}. In Section \ref{Sec-Further}, we present some further comments. In Section \ref{Sec-Open}, we end up discussing some open problems.

\section{Some admissible time-discretization schemes}\label{ExamplesTimeDisc}

In this section, we provide several classical time-discretization schemes that fit the setting \eqref{AbstractTimeDisc}--\eqref{TonEig} and satisfy assumptions \eqref{HypG-1}--\eqref{HypG-2}--\eqref{HypG-3}.

\subsection{The midpoint scheme}

	Perhaps the  simplest time-discretization of \eqref{FD-Abstract} which preserves the energy is the midpoint scheme: for $\tau >0$, the time-discrete equation is given by:
	\begin{equation}
		\label{Midpoint}
		\frac{y_\tau^{k+1} - y_\tau^k}{\tau} = \A \left(\frac{y_\tau^k + y_\tau^{k+1}}{2} \right), \quad k \in \Z, \qquad y_\tau^0 = y^0.
	\end{equation}

	Thus, if $\Phi_j$ is an eigenvector of $\A$ with eigenvalue $i \mu_j$, the solution $y_\tau$ of \eqref{Midpoint} with initial data  $y^0 = \Phi_j$ is given by 
	\begin{multline*}
		y_\tau^k = \left(\frac{1 + i \mu_j \tau /2}{1-i \mu_j \tau /2}\right)^k \Phi_j = \exp( i \mu_{j,\tau} k \tau) \Phi_j, 
		\\
		\hbox{ where $\mu_{j,\tau}$ is defined by }
		\exp( i \mu_{j,\tau} \tau ) = \frac{1 + i \mu_j \tau /2}{1-i \mu_j \tau /2},
	\end{multline*}
	yielding 
	$$
		 \mu_{j,\tau}  \tau = 2 \arctan( \mu_j \tau/2).
	$$

	Hence the midpoint scheme \eqref{Midpoint} fits the assumptions of Theorem \ref{Thm-Main} by setting $f(\alpha) = 2 \arctan(\alpha/2)$ and thus $R = \infty$.

\subsection{The fourth order Gauss Method}

	Let us present the so-called fourth order Gauss method  for discretizing \eqref{FD-Abstract}, which enters the frame of Runge-Kutta methods, see for instance \cite{Hairer-book}.

	It reads as follows:
	 \begin{equation}\label{Gauss}
		\left\{\begin{array}{ll}
 			\ds\kappa_i^k=\A\left(y_\tau^k+\tau\sum_{j=1}^2 \alpha_{ij}\kappa_j^k\right), \qquad i=1,2,
			\\
 			\begin{array}{ll}
 				\ds y_\tau^{k+1}=y_\tau^k+\frac{\tau}{2}(\kappa_1^k+\kappa_2^k), 
				\\
				\ds y_\tau^0 = y^0 \in \X \ \hbox{given},
			\end{array}  
			\qquad
			(\alpha_{ij}) =  
				\left(\begin{array}{cc}
  				\frac{1}{4}& \frac{1}{4}-\frac{\sqrt 3}{6}\\
  				\frac{1}{4}+\frac{\sqrt 3}{6}&\frac{1}{4}
  			\end{array}\right).
		  \end{array}\right.
	 \end{equation}
	An easy computation shows that, for any $k \in \Z$, 
	$$
			\frac{1}{2}\left(\kappa_1^k + \kappa_2^k \right)= \left(Id - \frac{\tau \A}{2} + \frac{\tau^2 \A^2}{12}   \right)^{-1} \A y_\tau^k, 
	$$
	which allows to rewrite \eqref{Gauss} as 
	$$
		y_\tau^{k+1} = \left(Id - \frac{\tau \A}{2} + \frac{\tau^2 \A^2}{12}   \right)^{-1} \left(Id + \frac{\tau \A}{2} + \frac{\tau^2 \A^2}{12}   \right) y_\tau^k, \quad k \in \Z.
	$$
	
	The spectral decomposition  of the semi-discrete scheme \eqref{Gauss} can easily be performed. If $y^0 = \Phi_j$, we obtain
 	$$
	    y_\tau^k = \exp\left(i \frac{f(\mu_j\tau)}{\tau} k \tau\right) y^0, \quad \hbox{ where }	f(\alpha)= 2 \arctan\left( \frac{6\alpha}{12-\alpha^2}\right).
	 $$
	Hence the discretization \eqref{Gauss} fits the assumption of Theorem \ref{Thm-Main} by setting
	$$
		f : (-2 \sqrt{3}, 2 \sqrt{3}) \to \R; \quad f(\alpha)= 2 \arctan\left( \frac{6\alpha}{12-\alpha^2}\right).
	$$
	Note that, here, the function $f$ is limited to the range where $R = 2\sqrt{3}$.

\subsection{The Newmark method for second order in time equations}

	The Newmark method is designed for second order in time equations such as, in particular, the wave equation. Namely, let $\A_0$ be a self-adjoint operator defined on an Hilbert space $\X_0$ with dense domain $\mathcal{D}(\A_0)$ and with compact resolvent, and consider the following equation:
	\begin{equation}
		\label{AbstractWaveCont}
		\varphi'' + \A_{0} \varphi = 0, \quad t \in \R, \qquad (\varphi, \varphi')(0) = (\varphi^0, \varphi^1) \in \mathcal{D}(\A_0^{1/2})\times \X_0.
	\end{equation}
	The Newmark method yields, for $\beta \in [0,1/4]$, the following time-discrete scheme:
	\begin{equation}
		\label{Anewmark}
		\left\{
			\begin{array}{ll}
 				\ds\frac{\varphi_\tau^{k+1}+\varphi_\tau^{k-1}-2\varphi_\tau^k}{\tau^2}
				+
				\A_0\left( \beta \varphi_\tau^{k+1}+(1-2\beta )\varphi_\tau^{k}+ \beta \varphi_\tau^{k-1}\right)=0,
			\smallskip
				\\
				\ds
 				\left( \frac{\varphi_\tau^0+\varphi_\tau^1}{2}, \frac{\varphi_\tau^1-\varphi_\tau^0}{\tau} \right) = (\varphi^0, \varphi^1)
  			\in \X_0^2.
 			\end{array}
		\right.
	\end{equation}
	System \eqref{Anewmark} is conservative and preserves the discrete energy:
	\begin{multline}\label{EnergyNewmark}
	E^{k+1/2} =  
		\norm{\A_0^{1/2} \left(\frac{\varphi_\tau^k + \varphi_\tau^{k+1}}{2} \right)  }_{\X_0}^2 
		+
	    \norm{\frac{\varphi_\tau^{k+1} - \varphi_\tau^k}{\tau} }_{\X_0}^2 
		\\
		+
 		(4 \beta - 1) \frac{\tau^2}{4} \norm{\A_0^{1/2} \left(\frac{\varphi_\tau^{k+1}- \varphi_\tau^k}{\tau} \right)}_{\X_0}^2.
	\end{multline}

	System \eqref{AbstractWaveCont} fits the abstract setting \eqref{FD-Abstract} by setting $\X = \mathcal{D}(\A_0^{1/2})\times \X_0$, where $\mathcal{D}(\A_0^{1/2})$ is endowed with the scalar product $\langle \A_0^{1/2}\cdot, \A_0^{1/2}\cdot \rangle_{\X_0} $ and 
	$$
		\A = \left(\begin{array}{cc} 0 & I \\ -\A_0 & 0 \end{array}\right).
	$$

	But another way to write \eqref{AbstractWaveCont} under the form \eqref{FD-Abstract} is to set $\X = \X_0^2$, 
	\begin{equation}
		\label{NewmarkFormulation}
		y_1 = \varphi' + i \A_0^{1/2} \varphi, \quad y_2 = \varphi' - i \A_0^{1/2} \varphi, 
	\qquad \A = \left(\begin{array}{cc} i \A_0^{1/2} & 0 \\ 0 & - i \A_0^{1/2} \end{array}\right).
	\end{equation}
	This formulation will be preferred to study \eqref{Anewmark} since the eigenvectors are now given as $\Phi_j^+ = (\Psi_j, 0)$, $\Phi_j^- = (0, \Psi_j)$ where $\Psi_j$ are the eigenvectors of $\A_0$, and $\A \Phi_j^{\pm} = \pm i \sqrt{\lambda_j} \Phi_j^{\pm}$, with $\A_0 \Psi_j  =  {\lambda_j} \Psi_j$. (At this step, remember that $\A_0$ is assumed to be a self-adjoint positive definite operator, so its spectrum is given by a sequence of positive real numbers going to infinity.)

	Now, setting
	$$
		\A_{0, \tau} = \left(I+ \left(\frac{4\beta-1}{4} \right) \tau^2 \A_0\right)^{-1} \A_0, 
	$$
	considering 
	\begin{equation}\label{VariableNewmark}
	    \left\{
		\begin{array}{ll}
		\ds    y_{1, \tau}^{k+1/2} = \frac{\varphi_\tau^{k+1}-\varphi_\tau^k}{\tau} + i \A_{0,\tau}^{1/2} \left( \frac{\varphi_\tau^k+\varphi_\tau^{k+1}}{2 }\right), 
		\smallskip
			\\
		\ds    y_{2, \tau}^{k+1/2} = \frac{\varphi_\tau^{k+1}-\varphi_\tau^k}{\tau} - i \A_{0,\tau}^{1/2} \left(
		\frac{\varphi_\tau^k+\varphi_\tau^{k+1}}{2}\right),
		 \end{array}\right.
	\end{equation}
	system \eqref{Anewmark} can be written as
	\begin{multline}
		\label{MidpointNewmark}
	\ds	\frac{y_\tau^{k+1/2} - y_\tau^{k-1/2}}{\tau} = \A_\tau \left(\frac{y_\tau^{k-1/2} + y_\tau^{k+1/2}}{2} \right), \quad k \in \Z,
	 		\\
			 \hbox{with } 
			\A_\tau = 
				\left(\begin{array}{cc}
					     i \A_{0,\tau}^{1/2} &  0
					\\
						0  & - i \A_{0,\tau}^{1/2}
				 \end{array}\right).
	\end{multline}
	Under this form, one easily sees that another energy for solutions of \eqref{Anewmark} is given by 
	\begin{equation}
		\tilde E^{k+1/2}  =  \frac{1}{2}\norm{y^{k+1/2}_\tau}_{\X_0^2}^2 
		= \norm{\frac{\varphi_\tau^{k+1}-\varphi_\tau^k}{\tau}}_{\X_0}^2 
		+ \norm{\A_{0, \tau}^{1/2}\left(\frac{\varphi_\tau^k+\varphi_\tau^{k+1}}{2}\right)  }_{\X_0}^2.
	\end{equation}
	As one easily checks, it turns out that the energies $\tilde E^{k+1/2}$ and $E^{k+1/2}$ (defined in \eqref{EnergyNewmark}) are equivalent when working within a filtered class $\mathfrak{C}(\delta/\tau)$  at scale $1/\tau$, thus making of no particular relevance to our purpose the fact that they do not coincide. 
	
	Now, one can easily show that if $y^{1/2} = \Phi$, where $\Phi$ is an eigenvector of $\A$ given by \eqref{NewmarkFormulation} corresponding to the eigenvalue $i \mu$, then the solution $y_\tau$ of \eqref{MidpointNewmark} is given by
	\begin{multline}
		\label{f-Newmark}
		y_\tau^{k+1/2} = \exp \left(i \frac{f(\mu\tau )}{\tau} k \tau \right) \Phi, \\
		 \hbox{ with } f(\alpha) = 2 \arctan\left(\frac{\alpha}{2} \frac{1}{ \sqrt{1+ (\beta-1/4)\alpha^2 }}\right), \quad R = \infty.
	\end{multline}
		
	Hence the Newmark approximation scheme \eqref{Anewmark} fits the assumptions of Theorem~\ref{Thm-Main}.
	
	Also note that the observation $\Bcal_1 \varphi + \Bcal_2 \varphi'$ for \eqref{AbstractWaveCont} can be discretized in two different ways: 
	\begin{itemize}
		\item 
		A natural discretization consists in taking
		\begin{equation}
			\label{Formulation1-a}
			\Bcal_1\left(\frac{\varphi_\tau^k + \varphi_\tau^{k+1}}{2}\right) + \Bcal_2 \left(\frac{\varphi_\tau^{k+1} - \varphi_\tau^k}{\tau} \right)
		\end{equation}
		for system \eqref{Anewmark}, which corresponds to 
		\begin{equation}
			\label{Formulation1}
		\hspace{-3ex}	\Bcal_\tau y_\tau^{k+1/2} := \frac{i}{2} \Bcal_1 \A_{0, \tau}^{-1/2} \left(y_{2, \tau}^{k+1/2} - y_{1, \tau}^{k+1/2}\right) + \frac{1}{2}\Bcal_2 \left(y_{1, \tau}^{k+1/2} + y_{2, \tau}^{k+1/2}\right)
		\end{equation}
		in the formulation \eqref{MidpointNewmark}.
		
		\item 
		A less natural discretization is as follows 
		\begin{equation}
				\label{Formulation2-a}
		\Bcal_1 \A_0^{-1/2} \A_{0, \tau}^{1/2}\left(\frac{\varphi_\tau^k + \varphi_\tau^{k+1}}{2}\right) + \Bcal_2 \left(\frac{\varphi_\tau^{k+1} - \varphi_\tau^k}{\tau} \right), 
	\end{equation}
		which corresponds to 
		\begin{equation}
				\label{Formulation2}
		\hspace{-3ex}		\Bcal y_\tau^{k+1/2} := \frac{i}{2} \Bcal_1 \A_{0}^{-1/2} \left(y_{2, \tau}^{k+1/2} - y_{1, \tau}^{k+1/2}\right) + \frac{1}{2}\Bcal_2 \left(y_{1, \tau}^{k+1/2} + y_{2, \tau}^{k+1/2}\right)
		\end{equation}
			in the formulation \eqref{MidpointNewmark}.		
	\end{itemize}
	
	Note that Theorem \ref{Thm-Main}, as stated, can only handle the second formulation \eqref{Formulation2-a} in which the observation operator does not depend on $\tau >0$, though it corresponds to a time-discrete observation operator of the form \eqref{Formulation2-a} for \eqref{Anewmark}, which seems less natural than \eqref{Formulation1-a}. 
	
	Whether or not system \eqref{Anewmark} is observable through \eqref{Formulation1-a} for general observation operators $\Bcal_1, \Bcal_2$ when the corresponding continuous system is observable is an open problem. 
	
	Though, if $\Bcal_2 =0$, a trick allows us to get the same result as in Theorem~\ref{Thm-Main} for an observation operator of the form
	$$
			\Bcal_1\left(\frac{\varphi_\tau^k + \varphi_\tau^{k+1}}{2}\right).
	$$
	Indeed, first apply Theorem \ref{Thm-Main} to the observation \eqref{Formulation2}. There we obtain, for $\delta >0$, $T$ as in \eqref{TimeConditionTau}, and initial data in $\mathfrak{C}(\delta/\tau)$ back in the variable $\varphi_\tau$: 
	\begin{multline*}
		\norm{\frac{\varphi_\tau^{1}-\varphi_\tau^0}{\tau}}_{\X_0}^2 
		+ \norm{\A_{0, \tau}^{1/2}\left(\frac{\varphi_\tau^0+\varphi_\tau^{1}}{2}\right)  }_{\X_0}^2 
		\\
		\leq C \tau \sum_{ k \tau \in (0,T)} \norm{\Bcal_1 \A_0^{-1/2} \A_{0, \tau}^{1/2}\left(\frac{\varphi_\tau^k + \varphi_\tau^{k+1}}{2}\right)}_\U^2.
	\end{multline*}
	Now, applying this identity to $\A_0^{1/2}\A_{0, \tau}^{-1/2} \varphi_\tau$, which is still a solution of \eqref{Anewmark} with initial data in $\mathfrak{C}(\delta/\tau)$, 
	\begin{multline*}
			\norm{\A_0^{1/2} \A_{0, \tau}^{-1/2}\left( \frac{\varphi_\tau^{1}-\varphi_\tau^0}{\tau}\right)}_{\X_0}^2 
			+ \norm{ \A_0^{1/2} \left(\frac{\varphi_\tau^0+\varphi_\tau^{1}}{2}\right)  }_{\X_0}^2 
			\\
			\leq C \tau \sum_{ k \tau \in (0,T)} \norm{\Bcal_1 \left(\frac{\varphi_\tau^k + \varphi_\tau^{k+1}}{2}\right)}_\U^2.
	\end{multline*}
	But easy spectral computations show that, in the class $\mathfrak{C}(\delta/\tau)$, there exists a constant $C>0$ depending only on $\delta$ such that
	$$
		\tilde E^{1/2} \leq C 	\left(\norm{\A_0^{1/2}\A_{0, \tau}^{-1/2}\left( \frac{\varphi_\tau^{1}-\varphi_\tau^0}{\tau}\right)}_{\X_0}^2 
			+ \norm{\A_0^{1/2} \left(\frac{\varphi_\tau^0+\varphi_\tau^{1}}{2}\right)  }_{\X_0}^2\right).
	$$

\section{A representation formula, properties of the kernel $\rho_\tau$ and applications}\label{Sec-Rho-tau}

We will first recall basic facts on the discrete Fourier transform. We will then prove Theorem \ref{Thm-DiscTransmutation} and give some estimates on the kernel $\rho_\tau$ in \eqref{KernelDiscTransmutation}, which we use in Section \ref{Sec-Proof-Main} to prove Theorem \ref{Thm-Main}.

\subsection{Discrete Fourier transforms}

Let us introduce the definition of the discrete Fourier and inverse Fourier transforms: 
\begin{definition}\label{DiscreteFourier}
	Given any function $u_\tau$ defined on $\tau \Z$, we define its discrete Fourier transform $\F_\tau[u_\tau]$ at scale $\tau$ as:
	 \begin{equation}
	        \label{FourierDis}
	        \F_\tau [ u_\tau](\mu_\tau) = \tau \sum_{k \in \mathbb{Z}} u_\tau(k \tau) \exp(- i \mu_\tau k \tau), \quad \mu_\tau \in (-\pi/\tau,\pi/\tau).
	  \end{equation}
	For any function $v \in L^2(-\pi/\tau, \pi/\tau)$, we define the inverse Fourier transform $\F_\tau^{-1}[v]$ at scale $\tau>0$ as:
	\begin{equation}
	        \label{FourierDisInv}
        		\F_\tau^{-1}[v](k \tau) = \frac{1}{2\pi} \int_{-\pi/\tau}^{\pi/\tau} v(\mu_\tau) \exp(i \mu_\tau k \tau) \ d\mu_\tau, \quad k \in \mathbb{Z}.
 	\end{equation}
\end{definition}

According to Definition \ref{DiscreteFourier}, one easily checks that these transforms are inverse one from another, so that,
\begin{equation}\label{FourierIso}
		\left\{ \begin{array}{l}
			\ds 
  			\F_\tau^{-1}[\F_\tau[u_\tau]] (k \tau ) = u_\tau (k \tau), \quad k \in \Z,
			\smallskip
			\\
			\ds
			\F_\tau[\F_\tau^{-1}[v]] (\mu_\tau) = v(\mu_\tau),\quad  \mu_\tau \in (-\pi/\tau,\pi/\tau), 
		\end{array}\right.
\end{equation}

Similarly as for the continuous Fourier transform, we also have the following Parseval identity:
\begin{equation}\label{ParsevalDis}
        	\frac{1}{2\pi} \int_{-\pi/\tau}^{\pi/\tau} |\F_\tau [u_\tau](\mu_\tau)|^2 \ d \mu_\tau = \tau \sum_{k\in\mathbb{Z}} |u_\tau(k\tau)|^2.
\end{equation}
These properties will be used in the sequel.

In the following, for a Hilbert space $H$, the space $L^2(\tau \Z; H)$ is the set of discrete functions $u_\tau$ defined on $\tau \Z$ with values in $H$ endowed with the norm
$$
	\norm{ u_\tau}_{L^2(\tau \Z; H)}^2 = \tau \sum_{k \in \Z} \norm{u_\tau(k \tau)}_H^2.
$$

\subsection{Proof of Theorem \ref{Thm-DiscTransmutation}}
 
\begin{proof}[Proof of Theorem \ref{Thm-DiscTransmutation}]
	Expand $y^0$ as 
	$$
		y^0 = \sum_{j, \ |\mu_j| \leq \delta/\tau} a_j \Phi_j.
	$$
	Then, for all $k \in \Z$, according to \eqref{TonEig}, 
	$$
		y_\tau^k = \sum_{j, \ |\mu_j| \leq \delta/\tau} a_j \Phi_j \exp (i \mu_{j,\tau} k\tau), 
	$$
	and thus $y(t)$ defined by \eqref{Transmutation} can be written as
	\begin{eqnarray}
		y(t) &=& \tau \sum_{k \in \Z} \rho_\tau(t, k \tau) \sum_{j, \ |\mu_j| \leq \delta/\tau} a_j \Phi_j \exp (i \mu_{j,\tau} k\tau)
		\nonumber
		\\
		& = & \sum_{j, \ |\mu_j| \leq \delta /\tau} a_j \Phi_j \tau \sum_{k \in \Z} \rho_\tau(t, k \tau) \exp(i \mu_{j, \tau} k \tau). 
		\nonumber
		\\
		& =& \sum_{j, \ |\mu_j| \leq \delta /\tau} a_j \Phi_j \F_\tau[ \rho_\tau(t, \cdot) ](- \mu_{j,\tau}).
		\label{XbyTransmutation}
	\end{eqnarray}
	
	According to \eqref{KernelDiscTransmutation}, for all $\mu_{\tau} \in (- \pi/\tau, \pi/\tau)$, 
	$$
		\F_\tau [\rho_\tau(t, \cdot)] (\mu_\tau) = \exp\left(-\frac{i g(\mu_\tau \tau) t}{\tau} \right) \chi\left(\mu_\tau \tau \right).		
	$$
	Due to the definition of $g$, for all $j$ such that $|\mu_j| \leq \delta/\tau$, $g(\mu_{j, \tau} \tau) = \mu_j\tau $. It then follows from \eqref{XbyTransmutation} that
	$$
		y(t) = \sum_{j, \ |\mu_j| \leq \delta /\tau} a_j \Phi_j \exp( i \mu_j t) \chi(\mu_{j,\tau} \tau).
	$$
	But the choice of $\chi$ implies that 
	$$
		y(t) = \sum_{j, \ |\mu_j| \leq \delta /\tau} a_j \Phi_j \exp( i \mu_j t).
	$$
	Hence $y(t)$ solves \eqref{FD-Abstract} with initial data $y^0$.
\end{proof}

\begin{remark}\label{Remark-a}
	When the parameter $R \in \R_+^* \cup \{\infty\} $ in \eqref{HypG-2} is such that $f(R) = \pi$ if $R$ is finite or $\underset{\infty}\lim f = \pi$, i.e. when $f$ is bijective from $(-R, R)$ to $(-\pi, \pi)$, then there is no need of introducing a cut-off function to get Theorem \ref{Thm-DiscTransmutation}. To be more precise, we have the following result:
	
	\begin{proposition}\label{Prop-DiscTransmutation-SansChi}
	Under the assumptions of Theorem \ref{Thm-DiscTransmutation}, if we further assume that $f$ is bijective from $(-R, R)$ to $(-\pi, \pi)$, then we can take the function $\chi$ in Theorem \ref{Thm-DiscTransmutation} to be identically one. In other words, if for $\tau >0$ we define $\rho_{\tau,0}(t,s)$ by
	\begin{equation}
		\label{KernelDiscTransmutation-SansChi}
		\rho_{\tau,0}(t, s) = \frac{1}{2\pi} \int_{- \pi/\tau}^{\pi/\tau} \exp\left(-\frac{i g(\mu_\tau \tau) t}{\tau} \right) e^{i \mu_\tau s}\, d\mu_\tau, \qquad (t,s) \in \R^2,
	\end{equation}
	we have the following result: for all $y^0 \in \mathfrak{C}(R/\tau)$ and $y^k_\tau$ is the corresponding solution of \eqref{AbstractTimeDisc},  the function $y(t)$ defined by \eqref{Transmutation} with $\rho_{\tau,0}$ instead of $\rho_\tau$ is the solution of \eqref{FD-Abstract} with initial data $y^0$.
	\end{proposition}

	Indeed, under the additional assumption that $f$ is bijective from $(-R, R)$ to $(-\pi, \pi)$, the function $g$ is defined on the whole interval $(-\pi, \pi)$. 
	
	Remark that all the numerical schemes presented in Section \ref{ExamplesTimeDisc} fit the assumptions of Proposition \ref{Prop-DiscTransmutation-SansChi} and thus Proposition \ref{Prop-DiscTransmutation-SansChi} applies for a wide range of numerical schemes.
	
	But even under this additional assumption, the localization properties of the kernel function $\rho_{\tau,0}$ may be very rough and are not suitable to derive good estimates on the time of uniform observability for solutions $y_\tau$ of \eqref{AbstractTimeDisc} as in Theorem \ref{Thm-Main}.
\end{remark}

\subsection{Localization of $\rho_\tau$}

We now analyze the function $\rho_\tau$ in \eqref{KernelDiscTransmutation}:

\begin{proposition}\label{RhoDescription}
	Let $f$ be a smooth function describing the time discrete operator $\T_\tau$ as in \eqref{TonEig}, assume \eqref{HypG-1}--\eqref{HypG-2}--\eqref{HypG-3}, and fix $\delta \in (0,R)$.

	Let $\varepsilon >0$ such that $\delta + \varepsilon < R$, and choose the function $\chi$ in Theorem \ref{Thm-DiscTransmutation} supported in $(- f(\delta + \varepsilon), f( \delta + \varepsilon))$ and real-valued.

	Let $T>0$ and $(t,s) \in (0,T) \times \R$ be such that 
	\begin{equation}
		\label{Cond(t,s)}
		t +\varepsilon  < s \inf_{|\alpha| \leq \delta + \varepsilon} \{f'(\alpha)\} \quad \hbox{or} \quad 	 s \sup_{|\alpha| \leq \delta + \varepsilon} \{f'(\alpha)\} < t- \varepsilon .
	\end{equation}
	Then for all $n \in \N$, there exists a constant $C_{n,\varepsilon}$ independent of $(t,s) \in (0,T) \times \R$ such that for all $(t,s)$ satisfying \eqref{Cond(t,s)}, 
	\begin{equation}
			\label{RhoSmall}
			|\rho_\tau(t,s) | \leq  \frac{C_{n,\varepsilon} \tau^{2n-1}}{\underset{|\alpha| \leq \delta + \varepsilon}\inf \{ |f'(\alpha) s- t | \}^{2n} },
	\end{equation}
	where $\rho_\tau$ is the kernel function given by \eqref{KernelDiscTransmutation}.
\end{proposition}

\begin{proof}
		Recall that $\rho_\tau$ is given by \eqref{KernelDiscTransmutation}. Hence
		\begin{equation}
			\label{KernelDiscTransmutation-2}
			\rho_\tau(t, s) = \frac{1}{2\pi \tau } \int_{- \pi}^{\pi} \exp\left(\frac{i}{\tau} \left(\alpha_\tau s - g(\alpha_\tau)t  \right)\right) \chi(\alpha_\tau) \, d\alpha_\tau.
		\end{equation}
		Remark then that
		\begin{multline*}
			- \frac{d^2}{d\alpha_\tau^2}\left( \exp\left(\frac{i}{\tau} \left(\alpha_\tau s - g(\alpha_\tau)t  \right)\right) \right)
			\\
			= \left(\frac{1}{\tau^2} (s- g'(\alpha_\tau) t)^2  + \frac{i}{\tau} g''(\alpha_\tau)t\right) \exp\left(\frac{i}{\tau} \left(\alpha_\tau s -g(\alpha_\tau)t \right)\right).
		\end{multline*}
				
		For $(t,s) \in (0,T) \times \R$ satisfying \eqref{Cond(t,s)} or, equivalently,
		\begin{equation}
			\label{Cond(t,s)-000}
			t \sup_{|\alpha_\tau| \leq f(\delta + \varepsilon)} \{ g'(\alpha_\tau)\} + \tilde \varepsilon < s \quad \hbox{or}  \quad s < t \inf_{|\alpha_\tau| \leq f(\delta + \varepsilon)} \{g'(\alpha_\tau)\} - \tilde \varepsilon,
		\end{equation}
		for some $\tilde \varepsilon>0$,
		the right hand-side of this identity does not vanish, and then we can write, for all $\alpha_\tau$ with $|\alpha_\tau| \leq f(\delta + \varepsilon)$, 
		\begin{multline}
			\label{StationaryTau}
			\ds \exp\left(\frac{i}{\tau} \left(\alpha_\tau s -g(\alpha_\tau)t \right)\right) = 	- \tau^2 G_\tau(\alpha_\tau)\frac{d^2}{d\alpha_\tau^2}\left( \exp\left(\frac{i}{\tau} \left(\alpha_\tau s - g(\alpha_\tau)t  \right)\right) \right)
			\\
			 \hbox{with } \ds G_\tau (\alpha_\tau) = \frac{1}{(s- g'(\alpha_\tau)t)^2 + i \tau g''(\alpha_\tau)t}.
		\end{multline}
		Hence, using the fact that $\chi$ is compactly supported in $(-f(\delta + \varepsilon), f(\delta+\varepsilon))$, we get
		\begin{align*}
				\lefteqn{\rho_\tau(t, s) = 
				\frac{1}{2\pi \tau } \int_{- \pi}^{\pi} 
					- \tau^2 G_\tau(\alpha_\tau) \frac{d^2}{d\alpha_\tau^2} \left(\exp\left(\frac{i}{\tau} \left(\alpha_\tau s - g(\alpha_\tau)t  \right)\right) \right)
					 \chi(\alpha_\tau) \, d\alpha_\tau}
					\nonumber
					\\
					 =& 
						- \frac{\tau}{2\pi } \int_{- \pi}^{\pi} 
							\exp\left(\frac{i}{\tau} \left(\alpha_\tau s - g(\alpha_\tau)t  \right)\right) 
							\frac{d^2}{d\alpha_\tau^2}\left(  G_\tau(\alpha_\tau) \chi(\alpha_\tau)\right) \, d\alpha_\tau.
					\nonumber
					\\
					 = &
						(-1)^{n}\frac{\tau^{2n-1}}{2\pi} \int_{- \pi}^\pi
							\exp\left(\frac{i}{\tau} \left(\alpha_\tau s - g(\alpha_\tau)t  \right)\right) 
							\left(\frac{d^2}{d\alpha_\tau^2}(  G_\tau(\alpha_\tau) \cdot )\right)^{n} \chi(\alpha_\tau) \, d\alpha_\tau,
									\label{KernelDiscTransmutation-3}
		\end{align*}
		where $ n \in \N$ and $\left(\frac{d^2}{d\alpha_\tau^2}(  G_\tau(\alpha_\tau) \cdot )\right)^{n}$ denotes the operator $\frac{d^2}{d\alpha_\tau^2}(  G_\tau(\alpha_\tau) \cdot )$ iterated $n$ times.
		
		We finally remark that, since $\chi$ is smooth and compactly supported on $(-f(\delta+\varepsilon), f(\delta + \varepsilon))$ and due to the explicit form of $G_\tau$ given by \eqref{StationaryTau}, for any $ n \in \N$, there exists a constant $C_{n,\varepsilon}$ such that for all $\alpha_\tau $ and $(t,s)$ satisfying \eqref{Cond(t,s)-000},
		$$
			\left|	\left( \frac{d^2}{d\alpha_\tau^2}(  G_\tau(\alpha_\tau) \cdot )\right)^{n} \chi(\alpha_\tau) \right| \leq \frac{C_{n,\varepsilon}}{(s- g'(\alpha_\tau )t)^{2n}}.
		$$
		This immediately yields \eqref{RhoSmall}.
\end{proof}

\subsection{The transmutation operator}

We then prove that the transmutation operator is bounded as an operator from $L^2(\tau \Z)$ in $L^2(\R)$. 
\begin{proposition}\label{Prop-TransmutOp}
	Let $f$ be a smooth function describing the time discrete operator $\T_\tau$ as in \eqref{TonEig}, assume \eqref{HypG-1}--\eqref{HypG-2}--\eqref{HypG-3}, and fix $\delta \in (0,R)$.

	Let $\varepsilon >0$ such that $\delta + \varepsilon < R$, and choose the function $\chi$ in Theorem \ref{Thm-DiscTransmutation} supported in $(- f(\delta + \varepsilon), f( \delta + \varepsilon))$ and real-valued.

	For $\tau>0$, set $\I_\tau$ the transformation defined for discrete functions $w_\tau$ compactly supported on $ \tau \Z$ with values in some Hilbert space $H$ by
	\begin{equation}
		\label{Itau}
			\I_\tau (w_\tau) (t) = \tau \sum_{k \in \Z} \rho_\tau(t, k \tau) w_\tau(k \tau),
	\end{equation}
	where $\rho_\tau$ is the kernel function given by \eqref{KernelDiscTransmutation}.

	Then the operator $\I_\tau$ is bounded from $L^2(\tau \Z, H)$ to $L^2(\R, H)$ and
	\begin{equation}
		\label{ItauBounded}
		\norm{\I_\tau}_{\mathfrak{L} (L^2(\tau \Z; H); L^2 (\R; H))} \leq \norm{\chi}_\infty  \sqrt{\sup_{|\alpha| \leq \delta+\varepsilon} f'(\alpha)}.
	\end{equation}
\end{proposition}

\begin{proof}
	For $w_\tau$ in $L^2 (\tau \Z; H)$ and $z \in L^2(\R, H)$, both compactly supported in time, using Fubini's formulas, we compute
	\begin{align*}
		\lefteqn{ 
			\int_\R \langle  \I_\tau (w_\tau)(t), z(t) \rangle_H \, dt 
		}
		\\
		& =   \int_\R  \left \langle \tau \sum_{k \in \Z} \frac{1}{ 2 \pi} \int_{-\pi/\tau}^{\pi/\tau} e^{ - i g(\mu_\tau \tau)t/\tau} \chi(\mu_\tau \tau) e^{i \mu_\tau k \tau} \, d\mu_\tau \, w_\tau(k \tau), z(t) \right\rangle_H\, dt
		\\
		& = \frac{1}{2 \pi } \int_{-\pi/\tau}^{\pi/\tau} \chi(\mu_\tau \tau) 
			\left\langle \tau \sum_{k \in \Z} w_\tau(k\tau)e^{i \mu_\tau k\tau} ,  \int_\R  z(t) e^{  i g(\mu_\tau \tau)t/\tau} \, dt
			\right\rangle_H \, d\mu_\tau
		\\
		& =  \frac{1}{2\pi}  \int_{-\pi/\tau}^{\pi/\tau}  \chi(\mu_\tau \tau)\left\langle \mathcal{F}_\tau[w_\tau](-\mu_\tau),  \widehat{ z}\left( -\frac{ g(\mu_\tau \tau)}{\tau} \right)\right\rangle_H \, d \mu_\tau
		\\
		&= \frac{1}{2\pi}  \int_{-\pi/\tau}^{\pi/\tau} \chi(-\mu_\tau \tau)\left\langle \mathcal{F}_\tau [w_\tau](\mu_\tau),  \widehat{ z}\left(\frac{ g(\mu_\tau \tau)}{\tau} \right)\right\rangle_H \, d \mu_\tau, 
	\end{align*}
	where $\widehat{z}$ denotes the Fourier transform of $z$. Hence, using \eqref{ParsevalDis},
	\begin{align*}
		\lefteqn{
		\left|
			\int_\R \langle  \I_\tau (w_\tau)(t), z(t) \rangle_H \, dt 
		\right|
		}
		\\
		& \leq 
		\left(\frac{1}{2\pi} \int_{-\pi/\tau}^{\pi/\tau} \norm{\mathcal{F}_\tau[w_\tau](\mu_\tau)}_H^2\, d\mu_\tau \right)^{1/2} 
		\\
		& \qquad \qquad 
		\left(\frac{1}{2\pi} \int_{-\pi/\tau}^{\pi/\tau} |\chi(-\mu_\tau \tau)|^2 \norm{\widehat{ z} \left(\frac{ g(\mu_\tau \tau)}{\tau} \right)}_H^2 \, d \mu_\tau\right)^{1/2}
		\\
		& \leq \norm{w_\tau}_{L^2(\tau \Z; H)} 
		\norm{\chi}_\infty  \sqrt{\sup_{|\alpha| \leq \delta + \varepsilon} f'(\alpha)}\norm{ z}_{L^2(\R; H)}.
	\end{align*}
	where the last line is justified by the change of variable $\mu \tau = g(\mu_\tau \tau)$, which is valid due to Assumption \eqref{HypG-3} and the fact that $\chi$ is supported in the interval $(-f(\delta +\varepsilon),f(\delta+\varepsilon))$:
	\begin{align*}
		\lefteqn{
		\frac{1}{2\pi} \int_{-\pi/\tau}^{\pi/\tau} |\chi(-\mu_\tau \tau)|^2 \norm{\widehat{ z} \left(\frac{ g(\mu_\tau \tau)}{\tau} \right)}_H^2 \, d \mu_\tau}
		\\
		& = 
		\frac{1}{2\pi} \int_{-(\delta+\varepsilon)/\tau}^{(\delta+\varepsilon)/\tau}|\chi(f(\mu \tau))|^2 \norm{\widehat{ z} (\mu)}_H^2 \, f'(\mu \tau) d \mu
		\\
		& \leq \norm{\chi}_\infty^2\Big(\sup_{|\alpha| \leq \delta + \varepsilon} f'(\alpha)\Big)\frac{1}{2\pi} \int_{-(\delta+\varepsilon)/\tau}^{-(\delta+\varepsilon)/\tau}\norm{\widehat{ z}(\mu)}_H^2 \, d\mu
		\\
		& \leq
		\norm{\chi}_\infty^2\Big(\sup_{|\alpha| \leq \delta + \varepsilon} f'(\alpha)\Big) \frac{1}{2 \pi} \int_\R  \norm{\widehat{ z}(\mu)}_H^2\, d \mu
		\\
		& 
		\leq 
		\norm{\chi}_\infty^2  \Big(\sup_{|\alpha| \leq \delta + \varepsilon} f'(\alpha)\Big) \norm{ z}_{L^2(\R; H)}^2.
	\end{align*}
	This concludes the proof of \eqref{ItauBounded}.
\end{proof}

\subsection{Proof of Theorem \ref{Thm-Main}}\label{Sec-Proof-Main}

\begin{proof}[Proof of Theorem \ref{Thm-Main}.]
	Let $y^0 \in \mathfrak{C}(\delta/\tau)$ and $y_\tau^k$ the corresponding solution of \eqref{AbstractTimeDisc}. Using Theorem \ref{Thm-DiscTransmutation}, and setting $y(t)$ as in \eqref{Transmutation} with $\chi$ compactly supported on $(-f(\delta +\varepsilon), f(\delta+ \varepsilon))$ for $\varepsilon >0$ small enough so that $\delta + \varepsilon < R$, we obtain the solution of \eqref{FD-Abstract} with initial data $y^0$.
	
	Using the observability \eqref{FD-Obs} of the continuous system \eqref{FD-Abstract}, we thus obtain 
	\begin{equation}\label{KeyTransmutation}
		\norm{y^0}_\X^2 \leq C \int_0^{T_0} \norm{ \Bcal y(t)}_\U^2 \, dt
		\leq C \int_0^{T_0} \norm{ \tau \sum_{k \in \Z} \rho_\tau(t, k \tau) \Bcal y_\tau^k}_\U^2 \, dt. 
	\end{equation}
	Estimate \eqref{KeyTransmutation} is the center of our argument. Now, we only have to check that the right hand-side of \eqref{KeyTransmutation} can be bounded by the right hand-side of \eqref{Fully-Obs}. 
	
	 To do this, we set 
	\begin{equation}
		\label{T-epsilon-1}
	T_{1, \varepsilon} = \frac{T_0} {\underset{|\alpha|\leq \delta + \varepsilon}\inf \{f'(\alpha) \}} + \varepsilon
	\quad \hbox{	and }\quad
	t_{1, \varepsilon} = - \varepsilon,
	\end{equation}
	and write
	\begin{equation}
		\label{RHS-Transmutation}
		\int_0^{T_0} \norm{ \tau \sum_{k \in \Z} \rho_\tau(t, k \tau) \Bcal y_\tau^k}_\U^2 \, dt \leq 3 \mathcal{O}_{\leq t_{1, \varepsilon}} + 3 \mathcal{O}_{t_{1, \varepsilon}, T_{1, \varepsilon} } + 3\mathcal{O}_{\geq T_{1, \varepsilon}}, 
	\end{equation}
	where 
	\begin{eqnarray}
			\mathcal{O}_{\leq t_{1, \varepsilon}} 
			& = &
			\int_0^{T_0} \Big\| \tau \sum_{k\tau \leq t_{1, \varepsilon}} \rho_\tau(t, k \tau) \Bcal y_\tau^k\Big\|_\U^2 \, dt,
			\label{ObsTrans-1}
			\\
			\mathcal{O}_{t_{1, \varepsilon}, T_{1, \varepsilon} }
			& = &
			\int_0^{T_0} \Big \| \tau \sum_{t_{1, \varepsilon} < k\tau < T_{1, \varepsilon}} \rho_\tau(t, k \tau) \Bcal y_\tau^k\Big\|_\U^2 \, dt, 
			\label{ObsTrans-2}
			\\
			\mathcal{O}_{\geq T_{1, \varepsilon}}
			& = &
			\int_{0}^{T_0} \Big\| \tau \sum_{k\tau \geq T_{1, \varepsilon}} \rho_\tau(t, k \tau) \Bcal y_\tau^k\Big\|_\U^2 \, dt.
			\label{ObsTrans-3}
	\end{eqnarray}
	
	Now, using Proposition \ref{RhoDescription} and Proposition \ref{Prop-TransmutOp}, we prove the following facts: 
	\begin{itemize}
		\item
	There exists a constant $C$ independent of $\tau >0$ such that
	\begin{equation}
		\label{EstObsTrans-2}
			\mathcal{O}_{t_{1, \varepsilon}, T_{1, \varepsilon} } \leq C \tau \sum_{t_{1, \varepsilon} < k\tau < T_{1, \varepsilon}} \norm{\Bcal y_\tau^k}_\U^2. 
	\end{equation}
		\item 
	For all $n \in \N^*$, there exists a constant $C_n = C_{n, \varepsilon,p,\delta}$ independent of $\tau >0$ such that
	\begin{equation}
		\label{EstObsTrans-1et3}
			\mathcal{O}_{\leq t_{1, \varepsilon}} +\mathcal{O}_{\geq T_{1, \varepsilon}}
			\leq C_n \tau^{4n-2 - 2p} \norm{y^0}_\X^2. 
	\end{equation}
	\end{itemize}
	
	Indeed, estimate \eqref{EstObsTrans-2} can be deduced immediately from Proposition \ref{Prop-TransmutOp} by choosing $w_\tau(k\tau) = \Bcal y^k_\tau$ for $k \tau \in (t_{1, \varepsilon}, T_{1, \varepsilon})$ and $0$ for $k\tau \notin  (t_{1, \varepsilon}, T_{1, \varepsilon})$.
	
	Let us then focus on \eqref{EstObsTrans-1et3}. According to Proposition \ref{RhoDescription}, we have
	\begin{equation}
			\mathcal{O}_{\leq t_{1, \varepsilon}} 
		 	\leq 
			\int_0^{T_0} \left(\sup_{k \in \Z} \norm{\Bcal y_\tau^k}_\U^2\right)
			\left(\tau \sum_{k\tau \leq t_{1, \varepsilon}} |\rho_\tau(t, k \tau)| \right)^2
			 \, dt,
	\end{equation}
	Using that the discrete semi-group \eqref{AbstractTimeDisc} is conservative in norms $\X$ and $\mathcal{D}(\A^p)$, preserves $\mathfrak{C}(\delta/\tau)$, and that $\Bcal$ satisfies \eqref{C-p-boundedness}, we immediately have that
	\begin{equation}
		\label{EstCxk}
		\sup_{k \in \Z} \norm{\Bcal y^k_\tau}_\U^2 \leq \frac{C_{p, \delta}^2}{\tau^{2p}} \norm{y^0}_\X^2. 
	\end{equation}
	Besides, according to Proposition \ref{RhoDescription}, for all $t \in (0, T_0  )$, and $n \geq 1$, 	
	\begin{multline}
	\tau \sum_{k\tau \leq t_{1, \varepsilon}} |\rho_\tau(t, k \tau)| 
		\leq 
		\tau  \sum_{k\tau \leq t_{1, \varepsilon}} \frac{C_{n} \tau^{2n-1}}{ \Big(\underset{|\alpha|\leq \delta + \varepsilon}\inf | f'(\alpha) k \tau - t|\Big)^{2n} }
		\\
		\leq
		\tau  \sum_{k\tau \leq t_{1, \varepsilon}} \frac{C_{n} \tau^{2n-1}}{\Big(t -  \underset{|\alpha|\leq \delta + \varepsilon}\inf  \{f'(\alpha)\} k \tau\Big)^{2n} }
		\leq	C_{n} \tau^{2n-1}, 
	\end{multline}
	for some constant $C_{n}$ depending on $n\geq 1$.

	Similar estimates can be done to bound $\mathcal{O}_{\geq T_{1, \varepsilon}}$, yielding \eqref{EstObsTrans-1et3} immediately.

	We now conclude the proof of Theorem \ref{Thm-Main}. Estimates \eqref{RHS-Transmutation} together with \eqref{EstObsTrans-2} and \eqref{EstObsTrans-1et3} show that
	\begin{multline}
		\int_0^{T_0} \norm{ \tau \sum_{k \in \Z} \rho_\tau(t, k \tau) \Bcal y_\tau^k}_\U^2 \, dt 
		\leq C \tau \sum_{t_{1, \varepsilon} < k\tau < T_{1, \varepsilon}} \norm{\Bcal y_\tau^k}_\U^2 \\
		+ C_n \tau^{4n-2-2p}  \norm{y^0}_\X^2. 
		\label{TransmutationAlmostOK0}
	\end{multline}
	Thus \eqref{KeyTransmutation} implies
	\begin{equation}
		\label{TransmutationAlmostOK}
		\norm{y^0}_\X^2 \left(1- C_n \tau^{4n-2-2p}   \right)
		\leq 
		C \tau \sum_{t_{1, \varepsilon} < k\tau < T_{1, \varepsilon}} \norm{\Bcal y_\tau^k}_\U^2.
	\end{equation}
	Since the discrete semigroup \eqref{AbstractTimeDisc} is conservative, we can shift the time in \eqref{TransmutationAlmostOK}, and obtain 
	\begin{equation}
		\label{TransmutationAlmostOK-1}
			\norm{y^0}_\X^2 \left(1- C_n \tau^{4n-2-2p} \right)
			\leq 
			C \tau \sum_{0 < k\tau < T_{1, \varepsilon}- t_{1, \varepsilon}} \norm{\Bcal y_\tau^k}_\U^2.
	\end{equation}
	Taking $n \geq 1 + p$ and $\tau >0$ small enough, we obtain \eqref{Fully-Obs} with 
	$$
		T_\varepsilon = T_{1, \varepsilon} - t_{1, \varepsilon} = \frac{T_0}{\underset{|\alpha|\leq \delta + \varepsilon}\inf \{f'(\alpha) \} } + 2 \varepsilon.
	$$
	This concludes the proof of Theorem \ref{Thm-Main} by taking $\varepsilon>0$ small enough.
\end{proof}

\section{A reverse formula and its applications}\label{Sec-Reverse}

\subsection{A reverse formula}

In this section, our aim is to explain that the representation formula derived in Theorem \ref{Thm-DiscTransmutation} allowing to write the solutions of the continuous abstract equation \eqref{FD-Abstract} as functions of those of the time-discrete one \eqref{AbstractTimeDisc} can be reversed:

\begin{theorem}\label{DiscTransmutationReverse}
	Let $f$ be a smooth function describing the time discrete operator $\T_\tau$ as in \eqref{TonEig}, assume \eqref{HypG-1}--\eqref{HypG-2}--\eqref{HypG-3}, and fix $\delta \in (0,R)$.

	Let $f$ be a smooth function describing the time discrete evolution as in \eqref{TonEig}, and assume \eqref{HypG-1}, \eqref{HypG-2} and \eqref{HypG-3}.

	Let $\chi$ be a $C^\infty$ function compactly supported in $(-R, R)$ and equal to $1$ in $(-\delta, \delta)$. For $\tau >0$, define then $q_\tau(t,s)$ for $(t, s) \in \R^2$ as 
	\begin{equation}
		\label{KernelDiscTransmutation-Rev}
		q_\tau(t, s) = \frac{1}{2\pi} \int_{ \R} \exp\left(\frac{i f(\mu \tau) s}{\tau} \right) \chi\left(\mu \tau \right) e^{ - i \mu t}\, d\mu.
	\end{equation}
	Then, if $y^0 \in \mathfrak{C}(\delta/\tau)$ and $y(t)$ is the corresponding solution of \eqref{FD-Abstract}, the function $y_\tau$ defined for $k \in \Z$ by 
	\begin{equation}
		\label{Transmutation-Rev}
		y^k_\tau = \int_\R q_\tau(t, k \tau) y(t) \, dt
	\end{equation}
	is the solution of \eqref{AbstractTimeDisc} with initial data $y^0$.
\end{theorem}

\begin{proof}
		The proof is similar to the one of Theorem \ref{Thm-DiscTransmutation}: Write the solution $y(t)$ of \eqref{FD-Abstract} on the basis of eigenfunctions of $\A$, and remark that for all $\mu \in \R$ satisfying $|\mu| \leq \delta/\tau $, due to the definition of $q_\tau$ in \eqref{KernelDiscTransmutation-Rev}, for all $s \in \R$,
		$$
			\int_\R q_\tau(t, s) \exp( i \mu t) \, dt = \chi(\mu \tau) \exp \left( \frac{i f(\mu \tau) s}{\tau}\right) = \exp \left( \frac{i f(\mu \tau) s}{\tau}\right).
		$$
		Details are left to the reader.
\end{proof}

\begin{remark}
	Similarly as in Remark \ref{Remark-a}, if $R$ in \eqref{HypG-2} satisfies $R = \infty$, Theorem \ref{DiscTransmutationReverse} still holds when taking $\chi \equiv 1$, which consists in replacing $q_\tau$ in \eqref{KernelDiscTransmutation-Rev} by  $q_{\tau,0}$ defined by
	$$
		q_{\tau,0}(t, s) = \frac{1}{2\pi} \int_{ \R} \exp\left(\frac{i f(\mu \tau) s}{\tau} \right) e^{ - i \mu t}\, d\mu.	
	$$
	Note that this definition has to be understood in the sense of $\mathcal{D}'(\R^2)$ as the integrand is not integrable.
\end{remark}

Similarly, one can prove the following:

\begin{theorem}\label{DiscTransmutationReverse-2}
With the notations and assumptions of Theorem \ref{DiscTransmutationReverse}, if $0<\delta_1 < \delta_2 <R$ and $y^0 \in \mathfrak{C}(\delta_2/\tau)\cap \mathfrak{C}(\delta_1/\tau)^\perp$, the same result holds for a smooth function $\chi$ compactly supported in $(-R, R)$ and equal to $1$ in $(-\delta_2, - \delta_1) \cup (\delta_1, \delta_2)$.
\end{theorem}

\begin{remark}
	The notation $\perp$ in $\mathfrak{C}(\delta_2/\tau)\cap \mathfrak{C}(\delta_1/\tau)^\perp$ stands for  the orthogonal complement with respect to the $\X$ scalar product, so that $\mathfrak{C}(\delta_2/\tau)\cap \mathfrak{C}(\delta_1/\tau)^\perp$ stands for the set $\hbox{Span}\{\Phi_j :  \delta_1 <\mu_j \tau \leq \delta_2  \}$. 
\end{remark} 

\begin{proof} 
	The proof of Theorem \ref{DiscTransmutationReverse-2} is the same as the one of Theorem \ref{DiscTransmutationReverse}. 
\end{proof}

As we will see below, these transmutation formula also yield non-trivial informations, for instance when dealing with uniform hidden regularity results or the optimality of the time-estimate \eqref{TimeConditionTau} in Theorem \ref{Thm-Main}.

\subsection{Uniform hidden regularity results}\label{Sec-Admissibility}

In this section, we are interested in the admissibility property - also called hidden regularity property - for system \eqref{FD-Abstract}. 

To be more precise, if $\Bcal \in \mathcal{L}(\mathcal{D}(\A^p), \U)$ for some $p\in \N$ and Hilbert space $\U$, system \eqref{FD-Abstract} is said to be admissible for $\Bcal$ if there exist a constant $C_0$ and a time $T_0>0$ such that any solution $y(t)$ of \eqref{FD-Abstract} with initial data $y^0 \in \mathcal{D}(\A^p)$ satisfies 
\begin{equation}
	\label{HiddenRegAss}
		\int_0^{T_0} \norm{\Bcal y(t)}_\U^2 \leq C_0 \norm{y^0}_\X^2.
\end{equation}

Note that, when $p = 0$, i.e. $\B \in \mathcal{L}(\X, \U)$, this property is straightforward since the energy of solutions $y$ of \eqref{FD-Abstract} is preserved. However, when $p>0$, this property is not at all granted and comes from subtle properties of the system under consideration, requiring  suitable assumptions on $\B$ and in particular an adequate interaction of $\B$ with the free dynamics generated by $\A$. The paradigmatic example of such situation concerns the wave equation in a bounded domain with homogenous Dirichlet boundary conditions observed through the flux on the boundary. In this case indeed, the operator $\Bcal$ is not bounded on $\X$ but is still admissible, see \cite{Lions}. This additional property then allows to define solutions in the sense of transposition for \eqref{ControlledSystem}, see e.g. \cite{TWbook}.

Let us also remark that by the semi-group property, it is straightforward to show that if the admissibility estimate \eqref{HiddenRegAss} is true for some $T_0$ with a constant $C_0$, it is true for all $T$ with constant $C(T) = C_0 (1+\lfloor T/T_0 \rfloor)$.

Finally, we point out that the estimate \eqref{HiddenRegAss} is the reverse of the observability estimate \eqref{FD-Obs}, and we may therefore expect that the strategy developed for getting uniform observability estimates for time-discrete approximations of \eqref{AbstractTimeDisc} also applies in the context of admissibility. This is indeed the case:

\begin{theorem}\label{MainFully5}
	Assume that $\Bcal \in \mathfrak{L}(\mathcal{D}(\A^p), \U)$ for some $p\in \N$ and $\U$ an Hilbert space, and that $\Bcal$ satisfies \eqref{C-p-boundedness} with constant $C_p$.
	
	Assume that equation \eqref{FD-Abstract} is admissible for $\Bcal$ at time $T_0$ with constant $C_0$, i.e. for all $y^0 \in \mathcal{D}(\A^p)$, the solution $y(t)$ of \eqref{FD-Abstract} with initial data $y^0$ satisfies \eqref{FD-Obs}.

	Let $f$ be a smooth function describing the time discrete operator $\T_\tau$ as in \eqref{TonEig}, assume \eqref{HypG-1}--\eqref{HypG-2}--\eqref{HypG-3}, and fix $\delta \in (0,R)$.
	
	Then, for all time $T>0$, there exists a constant $C(T)$ so that for all $\tau >0$ small enough, solutions $y_\tau$ of \eqref{AbstractTimeDisc} lying in $\mathfrak{C}(\delta/\tau)$ satisfy 
	\begin{equation}
		\label{HiddenRegDiscrete}
		\tau \sum_{k \tau \in (0,T)} \norm{\Bcal y^k}_\U^2 \leq C \norm{y^0}_\X^2.
	\end{equation}
\end{theorem}

Similarly as in Theorem \ref{Thm-Main}, Theorem \ref{MainFully5} is derived by careful estimates on the kernel function $q_\tau$ in Theorem \ref{DiscTransmutationReverse}.

Note however that Theorem \ref{MainFully5} can be found in \cite[Theorem 6.5]{je3}. There, it is proved using an equivalent form of the admissibility property \eqref{HiddenRegAss} in terms of packets of eigenfunctions, in the spirit of \cite{RTTT} for the observability of waves (see also \cite{TWbook}). Though, we will provide a short proof of Theorem \ref{MainFully5} by using the kernel $q_\tau$ given by Theorem \ref{DiscTransmutationReverse} to show the flexibility and efficiency of our strategy. The proof of Theorem \ref{MainFully5} is postponed to the end of the section.

We first show that the kernel function $q_\tau$ given by Theorem \ref{DiscTransmutationReverse} is mainly localized in some parts of $\R^2$ when we choose the function $\chi$ compactly supported in $(-\delta - \varepsilon, \delta + \varepsilon)$, for $\varepsilon \in (0, R- \delta)$:

\begin{proposition}\label{QDescription}
	Let $f$ be a smooth function describing the time discrete operator $\T_\tau$ as in \eqref{TonEig}, assume \eqref{HypG-1}--\eqref{HypG-2}--\eqref{HypG-3}, and fix $\delta \in (0,R)$. Let $\varepsilon >0$ such that $\delta + \varepsilon <R$, and choose the function $\chi$ in Theorem \ref{DiscTransmutationReverse} supported in $(-\delta - \varepsilon, \delta + \varepsilon)$.
	
	Let $T>0$ and $(t,s) \in \R_+ \times (0,T)$ be such that 
	\begin{equation}
		\label{Cond(t,s)-Q}
		t > s \sup_{|\alpha| \leq \delta + \varepsilon} \{ f'(\alpha)\} + \varepsilon \quad \hbox{or} \quad 	t <  s \inf_{|\alpha| \leq \delta + \varepsilon} \{f'(\alpha)\} -\varepsilon.
	\end{equation}
	Then for all $n \in \N$, there exists a constant $C_{n,\varepsilon}$ independent of $(t,s) \in \R \times (0,T) $ such that for all $(t,s)$ satisfying \eqref{Cond(t,s)-Q},
	\begin{equation}
			\label{QSmall}
			|q_\tau(t,s) | \leq  \frac{C_{n,\varepsilon} \tau^{2n-1}}{\underset{|\alpha| \leq \delta + \varepsilon}\inf \{ |f'(\alpha) s- t| \}^{2n} },
	\end{equation}
	where $q_\tau$ is the kernel given in \eqref{KernelDiscTransmutation-Rev}.
\end{proposition}

\begin{proof}
	Again, we only sketch the proof which can be done following the one of Proposition \ref{RhoDescription}.
	
	By a change of variable, similarly as in \eqref{KernelDiscTransmutation-2}, rewrite $q_\tau$ as 
	$$
		q_\tau(t,s) = \frac{1}{2\pi \tau} \int_{\R} \exp\left(\frac{i}{\tau}\left(f(\alpha) s - \alpha t \right) \right) \chi(\alpha) \, d\alpha.
	$$
	As in \eqref{StationaryTau}, we then remark that, for $(t, s) \in \R_+ \times (0,T)$ satisfying \eqref{Cond(t,s)-Q} and $|\alpha| \leq (\delta+ \epsilon)/\tau$,
	\begin{multline*}
	\ds \exp\left(\frac{i}{\tau} \left(f(\alpha) s -\alpha t \right)\right) = 	- \tau^2 F_\tau(\alpha)\frac{d^2}{d\alpha^2}\left( \exp\left(\frac{i}{\tau} \left(f(\alpha) s - \alpha t  \right)\right) \right)
	\\
	\hbox{with } \ds F_\tau (\alpha) = \frac{1}{(f'(\alpha) s- t)^2 - i \tau f''(\alpha)t}.
	\end{multline*}
	The rest of the proof follows line to line the one of Proposition \ref{RhoDescription}.
\end{proof}

Proposition \ref{Prop-TransmutOp} also has a counterpart:

\begin{proposition}\label{Prop-TransmutOp-Reverse}
	Let $f$ be a smooth function describing the time discrete operator $\T_\tau$ as in \eqref{TonEig}, assume \eqref{HypG-1}--\eqref{HypG-2}--\eqref{HypG-3}, and fix $\delta \in (0,R)$. Let $\varepsilon >0$ such that $\delta + \varepsilon <R$, and choose the function $\chi$ in Theorem \ref{DiscTransmutationReverse} supported in $(-\delta - \varepsilon, \delta + \varepsilon)$.

	For $\tau>0$, set $\mathcal{J}_\tau$ the transformation defined for functions $w(t)$ compactly supported on $\R$ with values in some Hilbert space $H$ by
	\begin{equation}
		\label{Itau-Q}
			\mathcal{J}_\tau (w) (k \tau) = \int_\R q_\tau(t,k \tau) w(t) \, dt,
	\end{equation}
	where $q_\tau$ is the kernel given in \eqref{KernelDiscTransmutation-Rev}.
	
	Then $\mathcal{J}_\tau$ is a bounded operator from $L^2(\R; H)$ to $L^2(\tau \Z; H)$:
	\begin{equation}
		\label{JtauBounded-Q}
		\norm{\mathcal{J}_\tau}_{\mathcal{L}(L^2(\R; H); L^2(\tau \Z; H))} \leq \frac{\norm{\chi}_{\infty}}  {\sqrt{ \underset{|\alpha| < \delta + \varepsilon}\inf f'(\alpha)}}.
	\end{equation}
\end{proposition}

Again, the proof is similar to the one of Proposition \ref{Prop-TransmutOp}. Details are left to the reader.

\begin{proof}[Proof of Theorem \ref{MainFully5}]
Theorem \ref{MainFully5} can be derived similarly as Theorem \ref{Thm-Main} by applying Theorem \ref{DiscTransmutationReverse}, Proposition \ref{QDescription} and Proposition \ref{Prop-TransmutOp-Reverse}.
\end{proof}

\subsection{Optimality of the time-estimate \eqref{TimeConditionTau}}\label{SectionOptimality}

Our goal here is to prove that the time-estimate \eqref{TimeConditionTau} is sharp. 

Let us begin with the following result:
\begin{theorem}\label{MainFully6}
	Let $f$ be a smooth function describing the time discrete operator $\T_\tau$ as in \eqref{TonEig}, assume \eqref{HypG-1}--\eqref{HypG-2}--\eqref{HypG-3}, and fix $\delta \in (0,R)$.
	
	Assume  that there exist $p \geq 0$ and a constant $C_{p}>0$ such that \eqref{C-p-boundedness} holds.

	Also assume that there exist a time $T_1$ and a constant $C$ such that for any $\tau >0$, solutions $y_\tau$ of \eqref{AbstractTimeDisc} lying in $\mathfrak{C}(\delta/\tau)$ satisfy \eqref{Fully-Obs}.	
	
	Then, for any $0 < \delta_1< \delta_2 < \delta$, for any time $T$ satisfying
	\begin{equation}
		\label{ReverseTime}
		T > T_1 \sup_{\alpha \in (\delta_1, \delta_2)} \{f'(\alpha) \}, 
	\end{equation}
	there exist positive constants $C$ and $\tau_0>0$ such that, for all $\tau\in (0,\tau_0)$, all solutions $y$ of \eqref{FD-Abstract} with initial data $y^0 \in \mathfrak{C}(\delta_2/\tau) \cap \mathfrak{C}(\delta_1/\tau)^\perp$ satisfy \eqref{FD-Obs}.
\end{theorem}

\begin{proof}
	Let $0 < \delta_1 < \delta_2 < \delta$, consider a smooth even function $\chi$ compactly supported on $(- \delta_2 - \varepsilon, - \delta_1 + \varepsilon) \cup (\delta_1 - \varepsilon, \delta_2 + \varepsilon)$ for some $\varepsilon \in (0, \delta - \delta_2)$, and set $q_\tau$ as in Theorem \ref{DiscTransmutationReverse-2}.
	
	Then, similarly as in Proposition \ref{QDescription}, one can prove that for $(t,s) \in \R \times (0,T_1)$ such that 
		\begin{equation}
			\label{Cond(t,s)-Q-2}
			t >s \sup_{\delta_1 - \varepsilon \leq |\alpha| \leq \delta_2 + \varepsilon} \{ f'(\alpha)\} + \varepsilon \quad \hbox{or} \quad 	t < s \inf_{\delta_1 - \varepsilon \leq |\alpha| \leq \delta_2 + \varepsilon} \{f'(\alpha)\} -\varepsilon,
		\end{equation}
		for all $n \in \N$, there exists a constant $C_{n}>0$ such that 
		\begin{equation}
				\label{QSmall-2}
				|q_\tau(t,s) | \leq  \frac{C_{n} \tau^{2n-1}}{\underset{\delta_1 - \varepsilon \leq |\alpha| \leq \delta_2 + \varepsilon} \inf \{ |f'(\alpha) s- t| \}^{2n} }.
		\end{equation}
	
		Similarly, Proposition \ref{Prop-TransmutOp-Reverse} still holds.
		
		Hence, following the proof of Theorem \ref{Thm-Main} and in particular of estimate \eqref{RHS-Transmutation}, we obtain that
		\begin{multline}
			\label{Eq-RHS-Transmutation-2}
			\tau \sum_{k \tau \in (0,T_1)} \norm{\Bcal y_\tau^{k}}_\U^2 \leq C \int_{-\varepsilon}^{T_1\sup_{\delta_1 - \varepsilon \leq |\alpha| \leq \delta_2 + \varepsilon} \{ f'(\alpha)\}+ \varepsilon} \norm{\Bcal y(t)}_\U^2 \, dt 
			\\+ C_{n} \tau^{4n-2-2p} \norm{y^0}_\X^2.
		\end{multline}
		Applying \eqref{Fully-Obs} and taking $\tau >0$ small enough, \eqref{FD-Obs} follows with $T = T_1 \sup_{\delta_1- \varepsilon\leq |\alpha| \leq \delta_2 + \varepsilon} f'(\alpha) + 2 \varepsilon$. 
		
		Since $\varepsilon$ can be chosen arbitrary small, the observability estimate \eqref{FD-Obs} holds for any solution of \eqref{FD-Abstract} with initial data in $\mathfrak{C}(\delta_2/\tau) \cap \mathfrak{C}(\delta_1/\tau)^\perp$ and time $T$ as in \eqref{ReverseTime}.
\end{proof}

We are now in position to prove that the estimate \eqref{TimeConditionTau} is sharp. In order to show this, we specify the abstract system \eqref{FD-Abstract} to the simplest case fitting the assumptions, namely the transport equation at velocity $1$ on the $1$-d circle denoted by $\mathbb{S}$ and identified with the interval $(0,1)$ with periodic boundary conditions. In this case, the equation reads:
\begin{equation}
	\label{TransportTorus}
	\partial_t y + \partial_x y = 0, \quad (t,x) \in \R \times \mathbb{S}.
\end{equation}
This correspond to an operator $\A = - \partial_x $, defined on $\X = L^2_{\#}(\mathbb{S})$, the space of periodic functions of period one in $L^2(0,1)$, with domain $\mathcal{D}(\A) = H^1_{\#}(\mathbb{S})$.

We consider the observation operator $\B y = y(0)$, which is continuous on $\mathcal{D}(\A) = H^1_{\#}(\mathbb{S})$ and takes value in $\U = \R$ .  

Since the solutions of the transport equations \eqref{TransportTorus} can easily be solved with characteristics, we get $y(t,x) = y^0(x-t)$. It is thus completely straightforward to show that
\begin{equation}
	\label{Obs-Torus}
	\int_0^1 |y^0(x)|^2 \, dx =  \int_0^1 |y(t,0)|^2 \, dt.
\end{equation}

In particular, applying Theorem \ref{Thm-Main} for some discretization scheme corresponding to $f$ satisfying \eqref{HypG-1}--\eqref{HypG-2}--\eqref{HypG-3}, taking $\delta <R$, for all $T$ satisfying
\begin{equation}
	\label{Time-Condition-Torus}
	T > \frac{1}{\underset{|\alpha| < \delta}\inf \{f' (\alpha)\} },
\end{equation}
there exist a constant $C$ and $\tau_0>0$ such that for all $\tau \in (0, \tau_0)$ and $y_\tau$ solution of
\begin{multline}
	\label{Torus-Discrete}
	y_\tau^{k+1} = \exp( i f(-i \A \tau)) y_\tau^k, \quad k \in \Z, \qquad y_\tau^0 = y^0, 
	\\ \hbox{ with } \A = - \partial_x, \quad \mathcal{D}(\A) = H^1_{\#}(\mathbb{S}), \quad \X = L^2_{\#}(\mathbb{S}),
\end{multline}
 with initial data $y^0 \in \mathfrak{C}(\delta/\tau)$, 
\begin{equation}
	\label{Obs-Discrete-Torus}
	\norm{y^0}_{L^2_\#(\mathbb{S})}^2 \leq C \tau \sum_{ k \tau \in (0,T)} |y_\tau^{k}(0)|^2.
\end{equation}

To prove the sharpness of the time estimate \eqref{Time-Condition-Torus}, we are thus going to show the following:

\begin{theorem}
	\label{Thm-Sharpness-Torus}
	Assume that $f$ satisfies \eqref{HypG-1}--\eqref{HypG-2}--\eqref{HypG-3} and take $\delta <R$.
	
	There is no time $T>0$ satisfying 
	\begin{equation}
		\label{T-too-small}
		T < \frac{1}{\underset{|\alpha| < \delta}\inf \{f' (\alpha)\} }
	\end{equation}
	and constants $C>0$ and $\tau_0>0$ such that for all $\tau \in (0,\tau_0)$, solutions $y_{\tau}$ of \eqref{Torus-Discrete} with initial data $y^0 \in \mathfrak{C}(\delta/\tau)$ satisfy \eqref{Obs-Discrete-Torus}.
\end{theorem}

\begin{proof}
	We argue by contradiction and assume that there exist a time $T$ satisfying \eqref{T-too-small} and positive constants $C>0$ and $\tau_0>0$ such that for all $\tau \in (0,\tau_0)$, solutions $y_{\tau}$ of \eqref{Torus-Discrete} with initial data $y^0\in \mathfrak{C}(\delta/\tau)$ satisfy \eqref{Obs-Discrete-Torus}.

	Since $T$ satisfies \eqref{T-too-small}, we can find $[\delta_1, \delta_2] \subset (0, \delta)$ such that 
	$$
		T < \frac{1}{\underset{\alpha \in [\delta_1,\delta_2]} \sup f'(\alpha)}.
	$$
	According to Theorem \ref{MainFully6}, choosing
	$$
		\tilde T \in \left( \sup_{\alpha \in [\delta_1, \delta_2]} f'(\alpha), 1 \right),
	$$
	we get the existence of a constant $C$ such that for all $\tau >0$ small enough, all solutions $y$ of \eqref{TransportTorus} with initial data $y^0 \in \mathfrak{C}(\delta_2/\tau)\cap \mathfrak{C}(\delta_1/\tau)^\perp$ satisfy
	\begin{equation}
		\label{Obs-Torus-Ttilde}
			\int_0^1 |y^0(x)|^2 \, dx \leq C \int_0^{\tilde T} |y(t,0)|^2\, dt.
	\end{equation}
	Now, we show that this cannot be true for $\tilde T <1$. In order to do this, let us remark that the spectrum of the operator $\A = -\partial_x$ with domain $\mathcal{D}(\A) = H^1_{\#}(\mathbb{S})$ on $\X = L^2_{\#}(\mathbb{S})$ simply is given by the Fourier basis $(\Phi_j (x) = \exp(2 i j \pi x))_{j \in \Z}$ and corresponds to the eigenvalues $(i \mu_j = 2 i j \pi)_{j \in \Z}$.
	
	Besides, using that solutions of \eqref{TransportTorus} are simply given by $y(t,x) = y^0(x-t)$,
	$$
		\int_0^{\tilde T} |y(t,0)|^2\, dt = \int_{1- \tilde T}^1 |y^0(x)|^2 \, dx,
	$$
	so that \eqref{Obs-Torus-Ttilde} can be rewritten as 
	\begin{equation}
		\label{Obs-Torus-Ttilde-bis}
			\int_0^1 |y^0(x)|^2 \, dx \leq C  \int_{1- \tilde T}^1 |y^0(x)|^2 \, dx.
	\end{equation}
	We thus have to prove that \eqref{Obs-Torus-Ttilde-bis} cannot be true uniformly with respect to $\tau >0$ for $y^0$ lying in $\mathfrak{C}(\delta_2/\tau)\cap \mathfrak{C}(\delta_1/\tau)^\perp$. 
	
	This can be proved by an explicit construction as follows. We choose $\delta_0 \in (\delta_1, \delta_2)$,  a smooth compactly supported function $\chi$ with support in $(-1,1)$ and with unit $L^2(\R)$-norm, and $x_0 \in \mathbb{S}$ such that $x_0 \in (0, 1 - \tilde T)$.
	
	We then set, for $x_0 \in \mathbb{S}$ to be chosen later on, 
	\begin{equation}
		\label{DisprovingY0}
		y_\tau^0(x) = \sum_{j \in \Z} \tau^{1/4} \chi\left( \sqrt{\tau} 2j \pi  - \frac{\delta}{\sqrt{\tau}} \right) \exp( 2 i j \pi (x-x_0)).
	\end{equation}
	First, let us note that the coefficients of $y_\tau^0$ in the basis $\Phi_j (x) = \exp(2i j \pi x)$ vanish for 
	\begin{equation}
		\label{Numbers-Of-Terms}
		\left| 2j \pi -\frac{\delta}{\tau}\right| \geq \frac{1}{\sqrt{\tau}}.
	\end{equation}
	This implies in particular that for all $\tau>0$, the sum in \eqref{DisprovingY0} is finite and thus makes sense, and that for $\tau >0$ small enough, $y_\tau^0$ indeed belongs to $\mathfrak{C}(\delta_2/\tau)\cap \mathfrak{C}(\delta_1/\tau)^\perp$. 

	We can also compute the $L^2_\#(\mathbb{S})$-norm of $y_\tau^0$ by Parseval's formula:
	\begin{multline}
		\label{NormL2-Y0}
		\int_0^1 |y^0_\tau(x)|^2 \, dx = \sqrt{\tau}  \sum_{j \in \Z} \left|
			\chi\left(\sqrt{\tau} 2 j \pi - \frac{\delta}{\sqrt\tau}\right)\right|^2 
			\\
			\underset{\tau \to 0}{\longrightarrow} 
			\frac{1}{2 \pi}\int_\R |\chi(\alpha)|^2 \, d\alpha = \frac{1}{2\pi},
	\end{multline}
	since the sum is a Riemann sum.
	
	We now claim that $y_\tau^0$ is concentrated around $x_0$. In order to show this, for a function $v =v (j)$ defined for $j \in \Z$, we introduce the discrete Laplacian
	$$
		\Delta_d v(j) = v(j+1) +v(j-1) - 2 v(j),
	$$
	and remark that
	$$
		- \Delta_d  \exp( 2 i j \pi (x-x_0)) = \exp( 2 i j \pi (x-x_0)) 4 \sin^2(\pi (x-x_0)).
	$$
	Thus, for $\varepsilon >0$ and $|x-x_0| \in (\varepsilon, 1- \varepsilon)$,
	\begin{align*}
		y_\tau^0(x) & = \sum_{j \in \Z} \tau^{1/4} \chi\left( \sqrt{\tau} 2j \pi  - \frac{\delta}{\sqrt\tau} \right) e^{2 i j \pi (x-x_0)}
		\\
		& = \frac{1}{ (4 \sin^2(\pi (x-x_0)))^n} \sum_{j \in \Z} \tau^{1/4} \chi\left( \sqrt{\tau} 2j \pi  - \frac{\delta}{\sqrt\tau} \right) (-\Delta_d)^n \left(e^{ 2 i j \pi (x-x_0)}\right) 
		\\
		& = \frac{1}{ (4 \sin^2(\pi (x-x_0)))^n} \sum_{j \in \Z} \tau^{1/4} (-\Delta_d)^n \left( \chi\left( \sqrt{\tau} 2j \pi  - \frac{\delta}{\sqrt\tau} \right) \right) e^{ 2 i j \pi (x-x_0)},
	\end{align*}
	so that
	\begin{equation}
		\label{Est-y-tau-0}
		\sup_{|x- x_0| \in (\varepsilon, 1-\varepsilon)} |y_\tau^0(x) | \leq C_{n,\varepsilon} \tau^{n-1/4},
	\end{equation}
	where we used that 
	$$
		\left| (-\Delta_d)^n \left( \chi\left( \sqrt{\tau} 2j \pi  - \frac{\delta}{\sqrt\tau} \right) \right)\right| \leq C_{n,\varepsilon} \tau^n, 
	$$
	and that the number of non-vanishing terms in the sum is of order $\tau^{-1/2}$, see \eqref{Numbers-Of-Terms}.
	
	In particular, considering \eqref{NormL2-Y0} and \eqref{Est-y-tau-0} with $\varepsilon>0$ such that $\{x \in \mathbb{S}, \, d_{\mathbb{S}}(x,x_0) \geq \varepsilon\}  \subset (1- \tilde T, 1)$, where $d_{\mathbb{S}}$ is the geodesic distance on $\mathbb{S}$, we have found a sequence $y_\tau^0 \in \mathfrak{C}(\delta_2/\tau)\cap \mathfrak{C}(\delta_1/\tau)^\perp$ such that 
	$$
		 \int_{1- \tilde T}^1 |y^0_\tau(x)|^2 \, dx \leq C_{n,\varepsilon}^2 \tau^{2n-1/2} \quad \hbox{ while } \quad \int_0^1 |y^0(x)|^2 \, dx \underset{\tau \to 0}{\longrightarrow} \frac{1}{2\pi},
	$$
	thus contradicting \eqref{Obs-Torus-Ttilde-bis}.
\end{proof}

\section{Further comments}\label{Sec-Further}

\subsection{Fully discrete approximation schemes}\label{Sec-Further-Space}

Our approach also applies in the context of fully-discrete approximation schemes for \eqref{FD-Abstract}, or more generally, to time-discrete approximations of a family of time continuous equations depending on a parameter. 

Following \cite{je3}, we introduce the following class:
\begin{definition}
	Let $p\in \N$, $C_p>0$, $C_0>0$ and $T_0>0$ be constant parameters and define the set $\mathcal{S}( p, C_p, C_0, T_0)$ as the set of elements $(\A, \X, \B, \U)$ such that:
	\begin{itemize}
		\item $\X$ and $\U$ are Hilbert spaces; 
		\item $\A$ is a skew-adjoint unbounded operator defined in $\X$ with dense domain $\mathcal{D}(\A)$ and compact resolvent;
		\item $\B \in \mathfrak{L}(\mathcal{D}(\A^p), \U)$ and satisfies \eqref{C-p-boundedness} with constant $C_p$;
		\item System \eqref{FD-Abstract} is observable through $\B$ in time $T_0$ and satisfies \eqref{FD-Obs} with constant $C_0$.
	\end{itemize}
\end{definition}

Theorem \ref{Thm-Main} can then be reformulated as follows:
\begin{theorem}
	\label{Thm-Main-Reformulated}
	Let $p \in \N$, $C_p >0$, $C_0 >0$ and $T_0 >0$. Let also $f$ be a smooth function describing the discretization process as in \eqref{TonEig} satisfying \eqref{HypG-1}--\eqref{HypG-2}--\eqref{HypG-3} and fix $\delta \in (0,R)$. Then, for all $T$ satisfying \eqref{TimeConditionTau}, there exist positive constants $C$ and $\tau_0>0$ such that for any $\tau \in (0,\tau_0)$, solutions $y_\tau$ of \eqref{AbstractTimeDisc} lying in $\mathfrak{C}(\delta/\tau)$ satisfy \eqref{Fully-Obs} uniformly for $(\A, \X, \B, \U)$ in $\mathcal{S}(p, C_p, C_0, T_0)$.
\end{theorem}

In particular, Theorem \ref{Thm-Main-Reformulated} allows to decompose the study of the observability properties of a fully-discrete approximation scheme in two steps. 

Indeed, if $(\A, \X, \B, \U)$ belongs to some $\mathcal{S}(p,C_p, C_0, T_0)$ and $\X$ is an infinite dimensional vector space, the usual strategy to construct a fully discrete approximation of $y' = \A y$ is to first design a space semi-discrete approximation scheme. If $h>0$ denotes the space semi-discretization parameter, these approximations are defined on finite-dimensional Hilbert spaces $\X_h$ and can be written as $y_h' = \A_h y_h$. Similarly, the observation, originally given by $\B y$ on the time and space continuous model $y' = \A y$, is approximated by $\B_h y_h$ for some operator $\B_h$ defined on $\X_h$ and with values in some Hilbert space $\U_h$ approximating $\U$ in some sense. 

Theorem \ref{Thm-Main-Reformulated} then states that, if we can find some $p \in \N$ and constants $C_p,\, C_0, \, T_0$ such that for all $h >0$, $(\A_h, \X_h, \B_h, \U_h) \in \mathcal{S}(p, C_p, C_0, T_0)$ then, taking $f$ satisfying the assumption \eqref{HypG-1}--\eqref{HypG-2}--\eqref{HypG-3} and fixing $\delta \in (0,R)$, the solutions $y_{\tau, h}$ of the fully discrete schemes 
$$
	y^{k+1}_{\tau, h} = \exp( i f(- i \A_h \tau)) y^k_{\tau, h}, \quad k \in \N, \qquad y^0_{\tau, h} = y^0_h
$$ 
satisfy 
$$
	\norm{y^0_h}_{\X_h}^2 \leq C \tau \sum_{ k \tau \in (0,T)} \norm{ \B_h y^k_{\tau, h}}_{\U_h}^2
$$
provided $y^0_h \in \mathfrak{C}_h(\delta/\tau)$, where $\mathfrak{C}_h(\delta/\tau)$ is the vector space spanned by the eigenfunctions of $\A_h$ corresponding to eigenvalues of modulus smaller than $\delta /\tau$.

Thus, in many practical situations, proving observability properties for fully-discrete approximations of \eqref{FD-Abstract} uniformly with respect to both space and time discretization parameters is reduced to proving observability properties for time continuous and space semi-discrete approximations of \eqref{FD-Abstract}.

We do not give further details on this strategy as we have already developed it in \cite[Section 5]{je3} - except for the estimate \eqref{TimeConditionTau} on the time which is new. 

Also note that our strategy also applies to derive uniform admissibility estimates for fully discrete approximations of \eqref{FD-Abstract} - where uniform means with respect to both space and time discretization parameters - provided uniform admissibility estimates are proved for the corresponding space semi-discrete and time continuous approximations of \eqref{FD-Abstract} - here, uniform means with respect to the space semi-discretization parameter.

In both situations, our approach successfully reduces the study of observability/admissibility issues for the time and space discrete approximations of \eqref{FD-Abstract} to the study of the observability/admissibility properties for the underlying time-continuous and space semi-discrete approximations of \eqref{FD-Abstract}, for which a large literature is available, see e.g. \cite{InfZua,CasMic,CasMicMunch,NegMatSch,NegZua} and the survey articles \cite{Zua05Survey,ErvZuaCime}.

\subsection{Discrete Ingham inequalities}\label{Sec-Further-Ingham}

As a by-product of our analysis, we proved the following discrete Ingham inequalities:

\begin{theorem}
	\label{Thm-Discrete-Ing}
	Let $I = \mathbb{N}$ or $\mathbb{Z}$ and $(\mu_j)_{j \in I}$ be an increasing sequence of real numbers such that, for some $\gamma >0$,
	\begin{equation}
		\label{GapCondition}
		\inf_{j \in I } | \mu_{j+1} - \mu_j | \geq \gamma.
	\end{equation}
	Let $f$ be a smooth function describing the time-discrete operator $\mathbb{T}_\tau$ as in \eqref{TonEig}, assume \eqref{HypG-1}--\eqref{HypG-2}--\eqref{HypG-3} and fix $\delta \in (0,R)$.
	
	Then for all time 
	\begin{equation}
		\label{TimeIngDiscrete}
		T> \frac{2\pi}{\gamma \underset{|\alpha| \leq \delta}\inf f'(\alpha) },
	\end{equation}
	there exist two positive constants $C$ and $\tau_0>0$ such that for all $\tau \in (0,\tau_0)$, for all $(a_j)_{j \in I} \in \ell^2(I)$ vanishing for $j \in I$ such that $|\mu_j| \tau \geq \delta$,
	\begin{equation}
		\label{InghamDiscrete}
		\frac{1}{C} \sum_{j \in I} | a_j|^2
		\leq
		 \tau \sum_{k \tau \in (0,T)} \left| \sum_{j \in I} a_j e^{i f(\mu_j \tau) k} \right|^2 
		\leq
		C  \sum_{j \in I} | a_j|^2.
	\end{equation}
\end{theorem}

Theorem \ref{Thm-Discrete-Ing} has to be compared with the results of \cite{NegZua06}, which derived a discrete Ingham lemma. Indeed, \cite{NegZua06} gives assumptions on an increasing sequence $(\lambda_j)_{j  \in I}$ of real numbers under which one can guarantee for all $\tau \in (0,\tau_0)$ and $(a_j)_{j \in I} \in \ell^2(I)$,
\begin{equation}
	\label{InghamDiscrete-Neg}
		\frac{1}{C} \sum_{j \in I} | a_j|^2
		\leq
		 \tau \sum_{k \tau \in (0,T)} \left| \sum_{j \in I} a_j e^{i \lambda_j k \tau} \right|^2 
		\leq
		C  \sum_{j \in I} | a_j|^2.
\end{equation}
It is proven in \cite{NegZua06} that \eqref{InghamDiscrete-Neg} holds when assuming the existence of a gap $\tilde \gamma >0$ such that
\begin{equation}
	\label{GapNeg}
	\inf_{j \in I } | \lambda_{j+1} - \lambda_j | \geq \tilde \gamma,
\end{equation}
and provided that the sequence $\lambda_j$ satisfies, for some $p \in (0,1/2)$, 
\begin{equation}
	\label{Cond-Et}
	\sup_{k, \ell} |\lambda_k - \lambda_\ell| \leq \frac{2\pi - \tau^p}{\tau}.
\end{equation}
When both \eqref{GapNeg} and \eqref{Cond-Et} hold, see \cite{NegZua06}, for any time $T> 2 \pi /\tilde \gamma$, there exists a constant $C >0$ depending only on $\tilde \gamma$ and $p$ such that \eqref{InghamDiscrete-Neg} holds. 

Of course, the result in \cite{NegZua06} is thus very similar to Theorem \ref{Thm-Discrete-Ing} as one can check by setting 
$$
	\lambda_j = \frac{f(\mu_j\tau)}{\tau}, \quad j \in I \hbox{ such that } |\mu_j| \tau \leq \delta, 
$$
which satisfies \eqref{GapNeg} with $\tilde \gamma =  \gamma \inf_{|\alpha| \leq \delta} f'(\alpha) $ and \eqref{Cond-Et} due to the assumption $\delta <R$.

Nevertheless, our approach yields various generalizations of discrete Ingham inequalities from the ones known in the continuous setting, for instance the ones in \cite{LoretiMehrenberger,Mehrenberger}.

For completeness, we briefly explain below how Theorem \ref{Thm-Discrete-Ing} follows from Theorem \ref{Thm-Main} and Theorem \ref{MainFully5}.

\begin{proof}[Proof of Theorem \ref{Thm-Discrete-Ing}]
	Let us consider $\X = \ell^2(I)$, set $\Phi_j = (a_{\ell,j})_{ \ell \in I}$ with $a_{\ell,j} = \delta_{\ell,j}$, and define the operators $\A$ and $\B$ by $\A \Phi_j = i \mu_j \Phi_j$, and $\B \Phi_j = 1$, which is continuous on $\mathcal{D}(\A)$ since $1/(1+ \mu_j^2)$ is summable under the assumption \eqref{GapCondition}. 
	
	According to Ingham's lemma \cite{Ing}, for all $T>2 \pi/\gamma$, there exists a constant $C >0$ such that for all $(a_j) \in \ell^2(I)$,
	$$
		\frac{1}{C} \sum_{j \in I} |a_j|^2 \leq \int_0^T \left| \sum_j a_j e^{ i \mu_j t} \right|^2\, dt \leq C \sum_{j \in I } |a_j|^2.
	$$
	This can be rewritten as follows: for all $y^0 = \sum_j a_j \Phi_j \in \ell^2(I)$, the solution $y$ of $y' = \A y$ with initial data $y^0$ satisfies
	$$
		\frac{1}{C} \norm{y^0}_{\ell^2(I)}^2 \leq \int_0^T | \B y(t) |^2 \, dt \leq C \norm{y^0}_{\ell^2(I)}^2. 
	$$
	Hence we can apply Theorem \ref{Thm-Main} and Theorem \ref{MainFully5} immediately to $\A$ and $\B$ and the corresponding discretizations described by  \eqref{TonEig} and satisfying assumptions \eqref{HypG-1}--\eqref{HypG-2}--\eqref{HypG-3}. Theorem \ref{Thm-Discrete-Ing} then follows from the explicit form of the solutions of $y_\tau^{k+1} = \T_\tau y_\tau^k$ with initial data $y_\tau^0 = \sum_j a_j \Phi_j$, which is simply given by $y_\tau^k = \sum_j a_j e^{i f(\mu_j \tau) k} \Phi_j$. 
\end{proof}

\subsection{Weak observability estimates}\label{Sec-Further-Weak}

In this section, we briefly focus on the case of weak observability estimates. To be more precise, we consider an observation operator $\B \in \mathfrak{L}(\mathcal{D}(\A^p), \U)$ and we assume the following: there exist a norm $\norm{\cdot}_*$, a time $T_w>0$ and a positive constant $C_{w}$ such that for all solutions $y$ of \eqref{FD-Abstract} with initial data $y^0 \in \mathcal{D}(\A^p)$,
\begin{equation}
	\label{WeakObs}
	\norm{y^0}_*^2 \leq C_w^2 \int_0^{T_w} \norm{\B y(t)}_\U^2\, dt.
\end{equation}

This property is a weak observability property for the system \eqref{FD-Abstract}. Roughly speaking, this property appears as soon as solutions $y$ of \eqref{FD-Abstract} satisfy the following unique continuation property:
\begin{equation}
	\label{UCP}
	\B y(t) = 0  \hbox{ on } (0,T) \Rightarrow y \equiv 0,
\end{equation}
since then one can simply define $\norm{\cdot}_*$ as 
\begin{equation}
	\label{Trivial-Norm-*}
	\norm{y^0}_*^2 = \int_0^T \norm{\B y(t)}_\U^2\, dt.
\end{equation}
Of course, defining $\norm{\cdot}_*$ as in \eqref{Trivial-Norm-*} does not provide any other information than the unique continuation property \eqref{UCP}.

When trying to quantify unique continuation properties, it is then natural to try to find a norm $\norm{\cdot}_*$ which can be compared to the norms constructed on, for instance, the iterated domains of the operator $\A$. 

To be more precise, we introduce the family $\X_r$ of Hilbert spaces indexed by $r \in \R$ as follows: for $n \in \N$, we set $\X_n = \mathcal{D}(\A^{n})$; for $r \in \R_+$, we take $n \in \N$ such that $r \in [n,n+1]$, and we define $\X_r$ as the interpolate between $\X_{n}$ and $\X_{n+1}$ of order $r - n$. We then define $\X_{r}$ for $r <0$ as the dual spaces of $\X_r$ ($\X_0= \X$ is identified with its dual). The corresponding norms $\norm{\cdot}_r$ on $\X_r$ are then the following ones: for $y = \sum a_j \Phi_j$,
$$
	\norm{y}_{r}^2 = \sum_j |a_j|^2 (1+ \mu_j^2)^{r}.
$$

We thus assume that there exist a constant $r \in \R$ and $C_r>0$, such that for all $y \in \X_p$,
\begin{equation}
	\label{Ass-norm-*}
	\norm{y}_r \leq C_r \norm{y}_*.
\end{equation}
Note that this definition makes sense for all $y \in \X_p$ since the weak observability estimate \eqref{WeakObs} guarantees that $\norm{y}_*$ is finite for $y \in \X_p$. Actually, estimate \eqref{WeakObs} and $\B \in \mathfrak{L}(\mathcal{D}(\A^p), \U)$ also imply $r \leq p$. Also note that when $r \geq 0$, inequality \eqref{WeakObs} is stronger than the one in \eqref{FD-Obs}. We are thus mainly interested in the case $r <0$. 

We then have the following variant of Theorem \ref{Thm-Main}:

\begin{theorem}\label{Thm-Main-Weak}
	Assume that $\Bcal \in \mathfrak{L}(\mathcal{D}(\A^p), \U)$ for some $p\in \N$ and $\U$ an Hilbert space, and that $\Bcal$ satisfies \eqref{C-p-boundedness} with constant $C_p$.
	
	Assume that equation \eqref{FD-Abstract} is observable through $\Bcal$ at time $T_w$ with constant $C_w$ and norm $\norm{\cdot }_*$, i.e. for all $y^0 \in \mathcal{D}(\A^p)$, the solution $y(t)$ of \eqref{FD-Abstract} with initial data $y^0$ satisfies \eqref{WeakObs}. Also assume that there exist a constant $C_r>0$ and $r \leq p$ such that \eqref{Ass-norm-*} holds.

	Let $f$ be a smooth function describing the time discrete operator $\T_\tau$ as in \eqref{TonEig}, assume \eqref{HypG-1}--\eqref{HypG-2}--\eqref{HypG-3}, and fix $\delta \in (0,R)$.
	Then, for all 
	\begin{equation}
		\label{TimeConditionTau-weal}
		T > \frac{T_w}{\underset{|\alpha| \leq \delta}\inf \{f'(\alpha) \}},
	\end{equation} 
	there exist positive constants $C$ and $\tau_0>0$ depending on $f$, $T$, $\delta$, $p$, $C_p$, $C_w$, $T_w$, $C_r$, $r$ such that for any $\tau \in (0,\tau_0)$, solutions $y_\tau$ of \eqref{AbstractTimeDisc} lying in $\mathfrak{C}(\delta/\tau)$ satisfy 
	\begin{equation}
		\label{Dis-Obs-Weak}
		\norm{y^0_\tau}_*^2 \leq C \tau \sum_{k \tau \in (0,T)} \norm{\B^* y^k_\tau}_\U^2.
	\end{equation}
\end{theorem}

\begin{proof}
	The proof of Theorem \ref{Thm-Main-Weak} is very close to the one of Theorem \ref{Thm-Main}, see Section \ref{Sec-Proof-Main}. Indeed, we first introduce $\varepsilon >0$, a function $\chi$ compactly supported on $(-f(\delta+\varepsilon), f(\delta+\varepsilon))$ and taking value one on $(-f(\delta), f(\delta))$ and the kernel function $\rho_\tau$ given by Theorem \ref{Thm-DiscTransmutation}. Then setting $t_{1, \varepsilon} = - \varepsilon$ and $T_{1, \varepsilon} = \varepsilon +  T_w/{\inf_{|\alpha| \leq \delta} \{f'(\alpha) \}}$, we get the counterpart of \eqref{TransmutationAlmostOK0}: for some constants $C, C_n$ independent of $\tau>0$,
	\begin{multline*}
		\int_0^{T_w} \norm{ \tau \sum_{k \in \Z} \rho_\tau(t, k \tau) \Bcal y_\tau^k}_\U^2 \, dt 
		\leq C \tau \sum_{t_{1, \varepsilon} < k\tau < T_{1, \varepsilon}} \norm{\Bcal y_\tau^k}_\U^2 \\
		+ C_n \tau^{4n-2-2p}  \norm{y^0}_\X^2,
	\end{multline*}
	
	Using then Theorem \ref{Thm-DiscTransmutation} and the fact that $y(t)$ defined by \eqref{Transmutation} is a solution of \eqref{FD-Abstract}, we have
	$$
		\norm{y^0}_*^2 \leq C \tau \sum_{t_{1, \varepsilon} < k\tau < T_{1, \varepsilon}} \norm{\Bcal y_\tau^k}_\U^2 \\
		+ C_n \tau^{4n-2-2p}  \norm{y^0}_\X^2,
	$$
	instead of \eqref{TransmutationAlmostOK}.
	
	One then uses that, since $y^0 \in \mathfrak{C}(\delta/\tau)$, there exists a constant $C>0$ independent of $\tau \in (0,1)$ and $y^0$ such that
	$$
		\norm{y^0}_{\X}^2 \leq \left\{
			\begin{array}{ll}
				\ds C \tau^{2r} \norm{y^0}_r^2 & \quad \hbox{ if } r <0, 
				\\
				\ds C \norm{y^0}_r^2 & \quad \hbox{ if } r \geq 0.
			\end{array}\right.
	$$
	Using then \eqref{Ass-norm-*}, estimate \eqref{TransmutationAlmostOK0} implies
	\begin{equation}
		\label{TransmutationAlmostOK-w}
		\norm{y^0}_*^2 \left(1- C_n \tau^{4n-2-2p+2\min\{r,0\} }   \right)
		\leq 
		C \tau \sum_{t_{1, \varepsilon} < k\tau < T_{1, \varepsilon}} \norm{\Bcal y_\tau^k}_\U^2.
	\end{equation}
	Thus, taking $n \geq 1 + p- \min\{r, 0\}$ and $\tau >0$ small enough, we obtain \eqref{Dis-Obs-Weak} with 
	$$
		T_\varepsilon = T_{1, \varepsilon} - t_{1, \varepsilon} = \frac{T_w}{\underset{|\alpha|\leq \delta + \varepsilon}\inf \{f'(\alpha) \} } + 2 \varepsilon.
	$$
	This concludes the proof of Theorem \ref{Thm-Main-Weak} by taking $\varepsilon>0$ small enough.
\end{proof}

Among the typical cases fitting our assumptions, let us quote the case of the wave equation on $(0,1)$ observed from a point $x_0 \in (0,1) \setminus \mathbb{Q}$. In that case, the equation reads:
\begin{equation}
	\label{1d-Wave}
	\left\{
		\begin{array}{ll}
			\partial_{tt}y - \partial_{xx} y = 0, \quad &\hbox{ for } (t,x) \in (0,T) \times (0,1),
			\\
			y(t,0) = y (t,1) = 0,\quad & \hbox{ for } t \in (0,T), 
			\\
			(y(0,\cdot), \partial_t y(0,\cdot)) = (y^0,y^1) &\hbox{ in } L^2(0,1) \times H^{-1}(0,1), 
		\end{array}
	\right.
\end{equation}
and the observation is given by $y(t,x_0)$.

Equation \eqref{1d-Wave} indeed fits the abstract setting of \eqref{FD-Abstract} by setting
\begin{multline*}
	Y = \left( \begin{array}{c} y \\ \partial_t y \end{array}\right) = \left( \begin{array}{c} Y_1 \\ Y_2 \end{array}\right), \quad 
	\A = \left( \begin{array}{cc} 0 & Id \\ \partial_{xx} & 0 \end{array}\right), 
	\\
	\hbox{ with } \X = L^2(0,1) \times H^{-1}(\Omega), \quad \mathcal{D}(\A) =H^1_0(0,1) \times L^2(0,1).
\end{multline*}
and the point-wise observation operator is given for smooth $Y$ by
$$
	\B Y = \B \left( \begin{array}{c} Y_1 \\ Y_2 \end{array}\right)  = Y_1(x_0).
$$
Sobolev's embedding easily shows that $\B$ is continuous from $\mathcal{D}(\A^{1/2+\varepsilon})$ to $\R$ for any $\varepsilon >0$.

Besides, expanding solutions $y$ of \eqref{1d-Wave} in Fourier, one easily checks
$$
	y(t,x) = \sqrt{2} \sum_{j \geq 1} \Big(a_j \exp(i j \pi t) + b_{j} \exp( - i j \pi t )\Big) \sin(j \pi x) , 
$$
where the coefficients $(a_j),\,(b_j)$ can be characterized from the expansion of the initial datum $(y^0, y^1)$: if $(y^0,y^1)$ are given by 
$$
	y^0 (x) = \sqrt{2} \sum_{ j \geq 1} \alpha_j \sin(j \pi x)	, \qquad y^1 (x) = \sqrt{2} \sum_{ j \geq 1} \beta_j \sin(j \pi x),
$$
the coefficients $(a_j),\,(b_j)$ are given by 
$$
	a_j = \frac{1}{2} \left( \alpha_j - i\frac{\beta_j}{j\pi}\right), \quad b_j = \frac{1}{2} \left( \alpha_j + i\frac{\beta_j}{j\pi}\right).
$$
In particular, using Parseval's identity, one easily gets
\begin{eqnarray*}
	\int_0^2 |\B Y(t) | \, dt = \int_0^2 |y(t,x_0)|^2 \, dt 
	 &=&  2 \sum_{j \geq 1} \left(|a_j|^2 +|b_j|^2\right) \sin^2(j \pi x_0)
	\\
	&	= &\sum_{ j \geq 1} \left( |\alpha_j|^2 + \frac{|\beta_j|^2}{j^2 \pi^2} \right) \sin^2(j \pi x_0).
\end{eqnarray*}
Of course, if $x_0 \notin \mathbb{Q}$, one easily checks that $\sin (j \pi x_0)$ cannot vanish for $j \in \N$, hence unique continuation holds. In particular, this implies that the semi-norm defined for 
$$
	Y = \left( \begin{array}{c} Y_1 \\ Y_2 \end{array}\right) = \left( \begin{array}{c}  \sqrt{2} \sum_{ j \geq 1} \alpha_j \sin(j \pi x) \\ \sqrt{2} \sum_{ j \geq 1} \beta_j \sin(j \pi x) \end{array}\right) 
$$
by 
$$
	\norm{ Y}_*^2 = \sum_{ j \geq 1} \left( |\alpha_j|^2 + \frac{|\beta_j|^2}{j^2 \pi^2} \right) \sin^2(j \pi x_0)
$$
is a norm satisfying \eqref{WeakObs} for $ T_w = 2$ and $C_w = 1$. 

Besides, one easily checks that the norms $\norm{\cdot}_r$ in that case are simply given by 
$$
	\norm{ Y}_r^2 = \sum_{ j \geq 1} \left( |\alpha_j|^2 + \frac{|\beta_j|^2}{j^2 \pi^2} \right) (1+(j \pi)^2)^{r/2}.
$$
Thus, if one wants to apply Theorem \ref{Thm-Main-Weak}, one should verify condition \eqref{Ass-norm-*}, i.e. that there exist $r >0$ and $C>0$ such that for all $j \in \N$,
\begin{equation}
	\label{Cond-x0}
	(1+(j \pi)^2)^{r/2} \leq C \sin^2(j \pi x_0).
\end{equation}
It turns out that condition \eqref{Cond-x0} is satisfied for a large set of irrational numbers $x_0 \notin \Q$, that we will denote by $\mathcal{S}$ in the following.

Indeed, $\mathcal{S}$ contains the irrational numbers $x_0 \in (0,1)$ whose expansion $[0,x_1,x_2, \cdots, x_n,\cdots]$ as a continuous fraction is given by a bounded sequence $(x_n)$, see \cite[p.23]{Lang}, for which the following property holds: there exists a positive constant $C$ such that for all $k \in \N$,
\begin{equation}
	\label{Cond-x0Lang}
	\inf_{p \in \Z} \{|q x_0 - p|\} \geq \frac{C}{q}.
\end{equation}
Since condition \eqref{Cond-x0Lang} is stronger than \eqref{Cond-x0}, we deduce in particular that $\mathcal{S}$ is uncountable.

Besides, $\mathcal{S}$ also contains all the irrational algebraic numbers according to Liouville's theorem: If $x_0 \in (0,1)$ is an algebraic number of degree $d$ on $\mathbb{Q}$, there exists $c>0$ such that for all $q \in \N$,
$$
	\inf_{ p \in \Z} \{|q x_0 - p|\} \geq \frac{c}{q^d}. 
$$

Theorem \ref{Thm-Main-Weak} applies when $x_0 \in \mathcal{S}$, and yields for instance the following observability result, corresponding to the Newmark method \eqref{Anewmark} with $\beta = 1/4$: for all $\delta >0$ and $T > 2(1+ \delta^2/4)$, there exists a constant $C>0$ such that for all $ \tau >0$ small enough, solutions $y_\tau$ of 
\begin{equation}
	\label{TimeDiscreteWave}
		\left\{
			\begin{array}{ll}
		\ds \frac{1}{\tau^2}\left(y_\tau^{k+1} - 2 y_\tau^k + y_\tau^{k-1} \right)- \partial_{xx} \left( \frac{1}{4}\left(y_\tau^{k+1} + 2 y_\tau^k + y_\tau^{k-1} \right)\right) = 0, 
		\\
		\hspace{6cm} \hbox{ for } (k,x) \in \Z \times (0,1), 
		\\
		\ds \frac{y^0_\tau + y^1_\tau}{2} = y^0,\,  \, \frac{y_\tau^1 - y_\tau^0}\tau = y^1,  \hbox{ for } x \in (0,1),
			\end{array}
		\right.
\end{equation}
with initial data 
$$
	(y^0, y^1) \in \left(\hbox{Span} \{ \sqrt{2} \sin(j \pi x), \hbox{ with } j \in \N \, \hbox{ satisfying } j \pi \leq \delta/\tau \} \right)^2
$$
all satisfy
\begin{equation}
	\norm{\left( \begin{array}{c} y_0 \\ y_1 \end{array}\right)}_*^2 \leq C \tau \sum_{k \tau \in (0,T)} |y^k_\tau(x_0)|^2.
\end{equation}

We conclude this paragraph by pointing out that similar weak observation properties often appear in the models for wave propagation on networks (\cite{DagZua06}). Nervertheless, condition \eqref{Ass-norm-*} is sometimes violated as in the case of the wave equation observed through a subset which does not satisfy the geometric control condition, see \cite{Robbiano95,Lebeau92}.

\section{Open problems}\label{Sec-Open}

\subsection{Non-conservative time-discretization schemes}

In this article, we focused on time-discretization schemes that preserve the energy of the solutions, which is a natural class since the continuous model \eqref{FD-Abstract} also preserves the energy of the solutions. 

However, in many situations, it may be interesting to consider dissipative numerical schemes, adding some numerical viscous effects to avoid instabilities. A very simple scheme introducing such dissipation properties is the Euler implicit method, which approximates \eqref{FD-Abstract} as follows:
\begin{equation}
	\label{EulerImp}
	\frac{y^{k+1}_\tau - y^k_\tau}{\tau} = \A y^{k+1}_\tau, \quad k \geq 0, \qquad y^0_\tau = y^0.
\end{equation}
If $y^0 = \Phi_j$, one easily checks that the solution $y^k_\tau$ is given by 
$$
	y^k_\tau = \left(\frac{1}{1- i \tau \mu_j}\right)^k \Phi_j
$$
Thus, when considering only eigenvectors such that $\tau |\mu| >\delta_0>0$, solutions $y^k_\tau$ decay exponentially. In particular, this implies that the observability inequality \eqref{Fully-Obs} may be not appropriate and should be replaced as follows: for all $y_\tau$ solutions of \eqref{EulerImp}, 
\begin{equation}
	\label{Final-Time-Obs}
	\norm{y_\tau^{\lfloor T/\tau \rfloor} }_\X^2 \leq C \tau \sum_{ k \tau \in (0,T)} \norm{\Bcal y_\tau^k}_\U^2.	
\end{equation}
We emphasize that, though the observability inequalities \eqref{Fully-Obs} and \eqref{Final-Time-Obs} are completely equivalent when the time-discretization schemes preserve the energy, it is not anymore the case when considering the Euler implicit method \eqref{EulerImp}, and estimate \eqref{Final-Time-Obs} is indeed weaker than the observability property \eqref{Fully-Obs} in that case.

This seems to indicate that a representation formula similar to \eqref{Transmutation} would involve kernels similar to the ones obtained to link the wave equation to the heat equations, see the ones proposed in \cite{Miller04a,ErvZuazuaARMA} for instance. 

Let us rapidly explain what is the new difficulty occurring when considering non-conservative schemes by transposing the formal arguments we developed in the introduction to the Euler implicit equation \eqref{EulerImp}. Similarly as in \eqref{Y-tau-Expanded}, solutions $y_\tau$ of \eqref{EulerImp} with initial datum $y^0$ as in \eqref{Init-Data} can be written as
$$
	y_\tau^k = \sum_{j} a_j \Phi_j \exp\left(i \frac{f( \mu_j \tau)}{\tau} k \tau\right),
$$
but this time $f$ is a complex-valued function given by
$$
	f(\alpha) = \arctan(\alpha) + \frac{i}{2} \log( 1+ \alpha^2). 
$$
Hence, similarly as in \eqref{Z-cont}, one can introduce the formal continuous version of $y_\tau$ given by 
$$
	z_\tau(s) = \sum_{j} a_j \Phi_j \exp\left(i \frac{f( \mu_j \tau)}{\tau} s \right).
$$
\emph{But} this solution is only well-defined for $s \geq 0$, and we are thus led to look for some kernel function $\tilde \rho_\tau = \tilde \rho_\tau(t,s)$ such that for all $\mu \in \R$ and $t \in R$,
\begin{equation}
	\label{Eq-GeneralCase}
	e^{i \mu t} = \int_{\R_+^*} \rho_\tau(t,s)\exp\left(i \frac{f( \mu \tau)}{\tau} s \right) \, ds,
\end{equation}
where the domain of integration has been modified into $\R_+^*$ instead of $\R$ as in \eqref{Formal-tilde-rho-req}. Also note that, to derive observability properties for the time-discrete model \eqref{EulerImp}, we actually need the identity \eqref{Eq-GeneralCase} only for $t $ in a bounded set of time and $|\mu| \leq \delta/\tau$, which might help deriving appropriate kernels. 

But we do not even know if equation \eqref{Eq-GeneralCase} is solvable. It may be the case but probably to the price of involving more singular kernels, as for instance in \cite{ErvZuazuaARMA} where the kernel was only defined for bounded sets of time, and allows to express solutions of the conservative model (wave-type model) in terms of solution of the dissipative models (heat-type model).

In some sense, this problem could also be thought as follows: given an observable conservative system, can we guarantee nice observability properties for viscous versions of it? With that respect, it is worth   pointing out the works \cite{CoronGuerrero,GlassJFA2010,LissyCras2012} considering the controllability properties of the transport equation
\begin{equation}
	\label{TransportVisc}
	\left\{\begin{array}{l}
			\partial_t y + \partial_x y -\varepsilon \partial_{xx} y = 0, \quad (t,x) \in (0,T) \times (0,1),
			\\
			y(t,0) = v(t), \quad y(t,1) = 0, \quad \hbox{ for } t \in (0,T),
		\end{array}\right.
\end{equation}
 with vanishing viscosity parameter $\varepsilon >0$, which illustrate the difficulties one encounters when considering control issues for vanishing viscosity systems (see also \cite{MicuRoventa} for another example). Indeed, to our knowledge, it is still not known what is the best time $T$ guaranteeing that the systems \eqref{TransportVisc} are uniformly (with respect to $\varepsilon >0$) null-controllable, though the critical time is expected to be $1$, i.e. the time needed to control the underlying transport equation obtained by setting $\varepsilon = 0$ and dropping the boundary condition at $x = 1$. 
 
 The complexity of this singular limit problem is one more evidence of the intrinsic complexity of passing from strongly dissipative dynamics to conservative ones. But the reverse problem is simpler. Indeed, in \cite{LopZZ} for instance (see also \cite{Phung02}), it was proved that the null-controllability and the observability of the heat equation 
 $$
 y_t - \Delta y =0,
 $$
 can be obtained as a limit, when $\varepsilon \to 0$,  of the corresponding wave-like properties for the one-parameter family of wave equations
 $$
 \varepsilon y_{tt} + y_t - \Delta y =0.
 $$

This is also in agreement with the results in \cite{Miller06a} allowing to write down solutions of the heat equation in terms of solutions of the wave equation. In our context, this corresponds to writing the solutions of the time-discrete implicit Euler schemes in terms of the time-continuous conservative dynamics. As one can easily check, this can be done with the same formula as in Theorem \ref{DiscTransmutationReverse}  This allows obtaining interesting results about the uniform (with respect of time-step) admissibility properties or the optimality of the possible uniform time-discrete observability results. But, unfortunately, the key issue of getting uniform observability results for the time-discrete dynamics out of the continues ones, requires a transformation expressing the solution of the conservative dynamics in terms of the dissipative one, and thus requires  further analysis.

 \subsection{Variable time-steps}

In many applications, it is important to allow the time discretization parameter to change in an adaptive manner. The precise study of such case seems to be out of reach by using our method which strongly relies on the use of discrete Fourier analysis.

\bibliographystyle{plain}

\end{document}